\documentclass[12pt,showkeys]{amsart}
\def\makeheadbox{{%
\hbox to0pt{\vbox{\baselineskip=10dd\hrule\hbox
to\hsize{\vrule\kern3pt\vbox{\kern3pt
\hbox{\bfseries Draft for discussion }
\hbox{Date of this version: 13.04.21}
\kern3pt}\hfil\kern3pt\vrule}\hrule}%
\hss}}}

\usepackage{color}
\usepackage{graphicx}
\usepackage{subfigure}
\usepackage{amssymb}
\usepackage{amsmath}
\usepackage{multirow}
\usepackage{enumerate}
\usepackage{epstopdf}
\usepackage{cite,stmaryrd,txfonts}
\usepackage{tikz}
\usetikzlibrary{snakes}

\usepackage{epic}
\usepackage{empheq}
\usepackage{cases}
 
\def\cequiv{\raisebox{-1.5mm}{$\;\stackrel{\raisebox{-3.9mm}{=}}{{\sim}}\;$}}

\newtheorem{theorem}{Theorem}[section]
\newtheorem{remark}[theorem]{Remark}\newtheorem{proposition}[theorem]{Proposition}
\newtheorem{lemma}[theorem]{Lemma}

\newtheorem{definition}[theorem]{Definition}

\newtheorem{ex}{Example}

\usepackage{stmaryrd}
\usepackage{tabularx}
\usepackage{tikz}
 
\newcounter{mnote}
\let\oldmarginpar\marginpar
\renewcommand\marginpar[1]{\-\oldmarginpar[\raggedleft\footnotesize #1]
  {\raggedright\footnotesize #1}}
  
\usepackage{footnote}
\usepackage{booktabs}
\newcommand{\ud}{\,d}

\usepackage{rotating}

\numberwithin{equation}{section}

\usepackage{cite}

\usepackage[colorlinks,linkcolor=black,anchorcolor=black,citecolor=black]{hyperref}
\usepackage[shortlabels]{enumitem} 
\setlist[enumerate]{nosep}
\usepackage{color, soul}

\usepackage{stmaryrd}

\begin{document}
\title{Lowest-degree robust finite element schemes for inhomogeneous bi-Laplace problems}

\author{Bin Dai}

\author{Huilan Zeng}

\author{Chensong Zhang}

\author{Shuo Zhang}



\maketitle

\begin{abstract}
In this paper, we study the numerical method for the bi-Laplace problems with inhomogeneous coefficients; particularly, we propose finite element schemes on rectangular grids respectively for an inhomogeneous fourth-order elliptic singular perturbation problem and for the Helmholtz transmission eigenvalue problem. The new methods use the reduced rectangle Morley (RRM for short) element space with piecewise quadratic polynomials, which are of the lowest degree possible. For the finite element space, a discrete analogue of an equality by Grisvard is proved for the stability issue and a locally-averaged interpolation operator is constructed for the approximation issue. Optimal convergence rates of the schemes are proved, and numerical experiments are given to verify the theoretical analysis. 
\end{abstract}

\tableofcontents

\section{Introduction} 

In this paper, we present a lowest-degree robust finite element schemes for inhomogeneous bi-Laplace problems. Here, by inhomogeneous bi-Laplace problems, we mean problems of inhomogeneous bi-Laplace operator which reads $\Delta\beta(x)\Delta$, where $\beta(x)>0$ is a varying coefficient. The operator can be used for, e.g., inhomogeneous thin plate. Meanwhile, it can be found in other applications such as the Helmholtz transmission eigenvalue problem. 
~\\

The main influence of the inhomogeneous coefficient lies on the establishment of the variational formulation. We take the homogeneous and inhomogeneous bi-Laplace equation for illustration. For the homogeneous bi-Laplace equation on a polygon $\Omega$,
\begin{equation}
\left\{
\begin{array}{rl}
\Delta^2u=f&\mbox{in}\ \Omega,
\\
u=\frac{\partial u}{\partial n}=0&\mbox{on}\ \partial\Omega,
\end{array}
\right.
\end{equation}
the generally used variational formulation is: find $u\in H^2_0(\Omega)$, such that 
\begin{equation}
(\nabla^2u,\nabla^2v)=(f,v)\ \ \forall\,v\in H^2_0(\Omega).
\end{equation}
This formulation is based on the  fundamental equality  by Grisvard \cite{Grisvard1992Singularities}
\begin{equation}\label{eq:Grisvard}
(\nabla^2w,\nabla^2v)=(\Delta w,\Delta v),\ \ \forall\,w,v\in H^2(\Omega)\cap H^1_0(\Omega),
\end{equation}
and the formulation is straightforward to be used for finite element discretization. We remark that the equality \eqref{eq:Grisvard} is a two-dimensional strengthened analogue of the Miranda--Talenti estimate which reads (\cite{maugeri2000elliptic})
\begin{equation}\label{eq:MTe}
\|\nabla^2 v\|_{0,\Omega}\leqslant \|\Delta v\|_{0,\Omega}, \quad v\in H^2(\Omega)\cap H^1_0(\Omega),\ \ \ \Omega\ \mbox{being\ convex}.
\end{equation}
This estimate \eqref{eq:MTe} can play important roles in applications(c.f., e.g., \cite{Smears.I;Suli.E,Neilan.M;Wu.M2019}). Note that this strengthened one \eqref{eq:Grisvard} holds on both convex and nonconvex domains.

On the other side, for the inhomogeneous bi-Laplace equation
\begin{equation}\label{eq:inbl}
\left\{
\begin{array}{rl}
\Delta\beta(x)\Delta u=f&\mbox{in}\ \Omega,
\\
u=\frac{\partial u}{\partial n}=0&\mbox{on}\ \partial\Omega,
\end{array}
\right.
\end{equation}
the variational formulation is: find $u\in H^2_0(\Omega)$, such that 
\begin{equation}\label{eq:vpih}
(\beta(x)\Delta u,\Delta v)=(f,v)\ \ \forall\,v\in H^2_0(\Omega).
\end{equation}
Based on \eqref{eq:Grisvard}, the well-posedness of \eqref{eq:vpih} follows directly. However, as \eqref{eq:Grisvard} does not generally hold for nonconforming finite element spaces and thus the inner product $\displaystyle\sum_T(\Delta v_h,\Delta w_h)$ is not coercive thereon, the low-degree discretization of \eqref{eq:inbl} may be difficult to establish.  
~\\

It is then an issue to construct low-degree finite element spaces whereon the discrete analogue of \eqref{eq:Grisvard} holds.  Recently, a finite element space $B^3_h$ for $H^2$ problems by piecewise cubic polynomials on general triangulations is presented by Zhang\cite{Zhang.S2021SCM}, namely,
\begin{multline*}
B_h^{3}:=\{v\in  \mathbb{P}^3_h:\ v\ \mbox{is\ continuous\ at}\ a\in\mathcal{N}_h;
\int_e\llbracket v\rrbracket=0,\ \mbox{and}\  \int_ep_e\llbracket\partial_{\bf n}v\rrbracket=0, \forall\, p_e\in P_1(e),\ \forall\, e\in\mathcal{E}_h^i\}
\end{multline*}
where we use $\llbracket \cdot\rrbracket$ for the jump across the edge $e$, and further,
$$
B_{ht}^{3}:=\{v\in B_h^{3}: v(a)=0,\ a\in\mathcal{N}_h^b;\ \int_ev=0,\   \forall\,e\in\mathcal{E}_h^b\},
$$
and a discrete analogue of \eqref{eq:Grisvard} is proved thereon. Here and in the sequel, for a subdivision of $\Omega$ by triangles or quadrilaterals, we use $\mathcal{N}_h$ for the set of all vertices, $\mathcal{N}_h=\mathcal{N}_h^i\cup\mathcal{N}_h^b$, with $\mathcal{N}_h^i$ and $\mathcal{N}_h^b$ comprising the interior vertices and the boundary vertices, respectively. Similarly, we use $\mathcal{E}_h=\mathcal{E}_h^i\bigcup\mathcal{E}_h^b$ for the set of all the edges, with $\mathcal{E}_h^i$ and $\mathcal{E}_h^b$ comprising the interior edges and the boundary edges, respectively.  Specifically, for $w_h,v_h\in B^3_{ht}$,
\begin{equation}\label{eq:b3grisvard}
\sum_T\int_T\nabla^2w_h\nabla^2v_h=\sum_T\int_T\Delta w_h\Delta v_h.
\end{equation} 
Used for homogeneous, inhomogeneous and Helmholtz transmission eigenvalue problems, the space can provide optimal discretization schemes\cite{Zhang.S2021SCM,xi2020high}. So far as we know, they are the up-to-date lowest-degree finite element scheme for inhomogeneous bi-Laplace problems. We remark that, besides the conforming and nonconforming finite element schemes, there may still be other existing kinds of discretizations devoted to the model problems, and we will not mention them too much in the present paper. 
~\\

In this paper, we present a lowest-degree robust finite element schemes for inhomogeneous bi-Laplace problems. Particularly, let $\Omega\subset\mathbb{R}^2$ be a simply-connected (convex or non-convex) polygon that can be subdivided to rectangular cells and $V_h^{\rm R}$ be the reduced rectangular Morley (RRM for short in the sequel) element space\cite{Shuo.Zhang2020} defined on $\Omega$ which is a space of certain piecewise quadratic polynomials ($P_2$), and we show that \eqref{eq:b3grisvard} holds on $V_{h0}^{\rm R}$, a subspace of $V_h^{\rm R}$ equipped with proper boundary condition. Then, for two basic model problems, more details of which will be given later, namely
\begin{itemize}
\item the fourth-order elliptic singular perturbation problem with varying coefficient:
\begin{equation}\label{eq:varyspp}
\varepsilon^{2} \Delta(\beta(\boldsymbol{x})\Delta u) -\Delta u = f;
\end{equation}
this equation models thin buckling plates with $u$ representing the displacement \cite{Frank1997};
\item the Helmholtz transmission eigenvalue problem
\begin{equation}\label{eq:eqhte}
\left(\triangle+k^2 \beta\right) \frac{1}{\beta-1}\left(\triangle+k^2\right) z=0;
\end{equation}
\end{itemize}
we show that the space $V^{\rm R}_{h0}$ provides robust optimal schemes on rectangular subdivisions. 
~\\

Evidently, when $\beta$ is a constant coefficient, \eqref{eq:varyspp} is just the standard fourth order elliptic perturbation problem which has been well studied. We refer to, e.g., \cite{Chen;Chen2014,Chen;Chen;Qiao2013,Chen;Chen;Xiao2014,Guzman;Leykekhman;Neilan2012,Nilssen;Tai;Winther2001,Tai;Winther2006,Wang;Wu;Xie2013,Wang;Shi;Xu2007,2ndWang;Shi;Xu2007,Xie;Shi;Li2010,Zhang;Wang2008, Chen;Liu;Qiao2010,Chen;Zhao;Shi2005,Wang2001,franz2014c0,semper1992conforming,vigo2006efficient,guzman2012family} for many different kinds of robust finite element schemes. All these schemes pretend to apply finite element methods which work for both fourth and second order problems. Those of them where \eqref{eq:b3grisvard} holds can be used directly for \eqref{eq:varyspp} with varying coefficients. Though, the space $B^3_{ht}$ of \cite{Zhang.S2021SCM} is so far the only nonconforming finite element space that admits \eqref{eq:b3grisvard} to be true. Meanwhile, there have been wide discussions on \eqref{eq:eqhte}, and we refer to, e.g., \cite{colton2010analytical,ji2014multigrid,ji2017nonconforming,xi2018c0ip,xi2020multi,yang2016mixed,yang2016non,camano2018convergence,geng2016c,ji2012algorithm,cakoni2013transmission,sun2011iterative} for part of them. Any finite element space that admits \eqref{eq:b3grisvard} to be true can be used for a finite element discretization of \eqref{eq:eqhte}, for which we refer to \cite{xi2020high} for an example. In this paper, we establish \eqref{eq:b3grisvard} on a space with theoretically lowest-degree polynomials, and construct schemes for the two model problems. Though not mentioned in the present paper, the more complicated structure of the varying coefficient, such as the multi-scale essence, may be studied more meticulously in future. New methods and analysis can be stimulated. As \eqref{eq:Grisvard} can be viewed as a stronger assertion than the Miranda-Talenti estimate \cite{maugeri2000elliptic}, the RRM element space can be used where a discrete Miranda-Talenti is needed. Further, the RRM element space is potentially able to be generalized to three and higher dimensions; this will be discussed more in future. 
~\\

It is notable that the RRM element space $V^{\rm R}_h$, originally given in \cite{Shuo.Zhang2020} and then studied in \cite{ZZZ2021} does not coincide with a ``finite element" defined by Ciarlet's triple\cite{Ciarlet1978,Brenner;Scott2007,Wang.M;Shi.Z2013mono}. Though, the schemes based on $V^{\rm R}_h$ are still practical computational ones by figuring out the basis functions whose supports are quite local. Actually, it can be proved that the minimal support of a basis function of $V^{\rm R}_{h0}$ is as Figure \ref{fig:3x3basis}, and a same cell can be covered by the supports of 25 basis functions of such type. Therefore, the restrictions of these 25 minimally-supported functions on this cell can not be linearly independent. This confirms that, the space $V^{\rm R}_{h0}$, as well as $V^{\rm R}_h$, can not be correspondent to a ``finite element" in Ciarlet's sense. On the other hand, the standard method to construct the approximation error which works for Ciarlet's finite elements can not be straightforward used for $V^{\rm R}_h$, and in \cite{Shuo.Zhang2020} and \cite{ZZZ2021}, the approximation error estimations are established by indirect ways other than constructing an interpolator directly. The method in \cite{Shuo.Zhang2020} works for estimation of broken $H^2$ norm and the method in \cite{ZZZ2021} works for convex domains. In the present paper, we reconstruct the estimation for both the broken $H^{2}$ and $H^{1}$ norms on both convex and non-convex domains. Inspired by the construction of quasi-spline interpolation operators in the spline function theory (see, for example, \cite{Wang;Lu1998,Sablonniere.P2003,2Sablonniere.P2003}), we propose a locally-averaging operator which preserves $P_{2}$ polynomials locally and is stable in terms of relevant Sobolev norms. Consequently, optimal error estimate of the interpolation operator is established. This interpolation operator is suitable for any regions that can be subdivided into rectangles, and particularly an optimal estimation can be given for the RRM element space in the broken $H^{1}$ norm on non-convex domains.  Therefore, the convergence analysis of the RRM element for the model problem~\eqref{eq:model problem} robust to $\varepsilon$ is derived. We remark that the newly-designed interpolation operator is not a projection onto the RRM element space, namely, it can not preserve every function in this space. A proof that the RRM space does not allow a locally-defined projective interpolation can be found in \cite{ZZZ-AML}, where a long-standing open problem is addressed. 
~\\

The rest of the paper is organized as follows. In Section \ref{sec:pre}, we present some necessary preliminaries. In Section~\ref{sec:rrmscheme}, the reduced rectangular Morley element space is revisited, the equality of \eqref{eq:b3grisvard} type is constructed in Lemma \ref{lem: discrete_property_1}, and the approximation error estimation is conducted based on a locally-averaging interpolation operator in Theorem \ref{thm:approxH02}. In Section~\ref{sec:convergence}, the convergence analysis of the RRM element scheme for the model problems are provided. Finally, in Section~\ref{sec:experiments}, some numerical experiments are given to verify the theoretical analysis. Through this paper, for the ease of the readers, we would postpone some long technical proofs to the appendix section.

\section{Preliminaries}
\label{sec:pre}

%

Throughout this paper, $\Omega \subset \mathbb{R}^{2}$ is a simply-connected (not necessarily convex) polygon, which can be subdivided into rectangles. We use $\nabla$ and $\nabla^{2}$ to denote the gradient operator and Hessian, respectively. We use standard notation on Lebesgue and Sobolev spaces, such as $L^{p}(\Omega)$, $H^{s}(\Omega)$, and $H^{s}_{0}(\Omega)$.  Denote by $H^{-s}(\Omega)$ the dual spaces of $H^{s}_0(\Omega)$. We use $\|\cdot\|_{s,\Omega}$ and $|\cdot|_{s,\Omega}$ for the standard Sobolev norm and semi-norm, respectively~\cite{Hughes1987}. We utilize the subscript $``\cdot_h"$ to indicate the dependence on grids. Particularly, an operator with the subscript $``\cdot_h"$ implies the operation is done cell by cell. Finally, $\lesssim$, $\gtrsim$, and $\cequiv$ respectively denote $\leqslant$, $\geqslant$, and $=$ up to a generic positive constant \cite{J.Xu1992}, which  might depend on the shape-regularity of subdivisions, but not on the mesh-size~$h$  and the perturbation parameter $\varepsilon$ in \eqref{eq:model problem}.

\subsection{Inhomogeneous fourth order elliptic perturbation problem}
The fourth-order elliptic singular perturbation problem is to find $u$ such that

\begin{equation}\label{eq:model problem}
\left\{
\begin{array}{rl}
 \varepsilon^{2} \Delta(\beta(\boldsymbol{x})\Delta u) -\Delta u = f, & \mbox{ in } \ \Omega , \\
u = \frac{\partial u}{\partial \mathbf{n}} = 0, & \mbox{ on } \ \partial\Omega ,
\end{array}
\right.
\end{equation}
where $\frac{\partial u}{\partial \mathbf{n}}$ denotes the normal derivative along the boundary $\partial \Omega$, $0< \varepsilon \ll 1$ is a real parameter, $\beta(\boldsymbol{x})$ is a bounded smooth non-constant function and $\beta(\boldsymbol{x}) \geqslant \beta_{\min }>0$. This equation models thin buckling plates with $u$ representing the displacement of the plate\cite{Frank1997}.

The corresponding variational formulation is given by : Find $u \in V:=H_0^2(\Omega)$ satisfying
\begin{equation} \label{eq: pertubation_variational_form}
\varepsilon^2 a(u, v)+b(u, v)=(f, v), \quad \forall v \in H_0^2(\Omega),
\end{equation}
where
$$
a(u, v)=\int_{\Omega} \beta(\boldsymbol{x})\Delta u \Delta v d x d y \text { and } b(u, v)=\int_{\Omega} \nabla u \cdot \nabla v d x d y .
$$
We define an energy norm on $V$ relative to $\varepsilon$ as
$$
\|w\|_{\varepsilon, \Omega}:=\sqrt{\varepsilon^2 |w|^2_{2,\Omega}+|w|^2_{1,\Omega}}.
$$
The well-posedness of \eqref{eq: pertubation_variational_form} is the classical result. 

\subsection{Helmholtz transmission eigenvalue problem}
The transmission eigenvalue problem is to find $k\in \mathbb{C}, \phi,\varphi\in H^2(\Omega)$ such that
\begin{equation} \label{eq: Helmholtz_transmission_model}
\left\{\begin{aligned}
&\Delta\phi + k^2\beta(\boldsymbol{x})\phi = 0, & \text { in } \Omega, \\
&\Delta\varphi  + k^2\varphi = 0, & \text { in } \Omega, \\
&\varphi - \phi  = 0, & \text { on } \partial \Omega, \\
&\frac{\partial \phi}{\partial \boldsymbol{n}} - \frac{\partial \varphi}{\partial \boldsymbol{n}} = 0, & \text { on } \partial \Omega,
\end{aligned}\right.
\end{equation}
where $\beta(\boldsymbol{x})$ is the index of refraction and $\boldsymbol{n}$ is the unit outward normal to the boundary $\partial \Omega$. Typically, it's assumed that $\beta(\boldsymbol{x})>1$ or $0<\beta(\boldsymbol{x})<1$.

Following the same procedure in \cite{ji2012algorithm}, let $u=\phi - \varphi  \in H_0^2(\Omega)$, then we obtain
$$
\left(\triangle+k^2\right) u=-k^2(\beta - 1) \phi .
$$
Dividing $\beta-1$ and applying $\left(\triangle+k^2 \beta\right)$ to both sides of the above equation, we obtain
$$
\left(\triangle+k^2 \beta\right) \frac{1}{\beta-1}\left(\triangle+k^2\right) u=0 .
$$
The transmission eigenvalue problem can be stated as: find $\left(k^2 \neq 0, u\right) \in \mathbb{C} \times H_0^2(\Omega)$ such that
$$
\int_{\Omega} \frac{1}{\beta-1}\left(\triangle u+k^2 u\right)\left(\triangle v+k^2 \beta v\right) d x=0,
$$
for all $v \in H_0^2(\Omega).$
Define
$$
\mathcal{A}_\tau(u, v) =\left(\frac{1}{\beta-1}(\triangle u+\tau u),(\triangle v+\tau v)\right)+\tau^2(u, v),
$$
$$
\mathcal{B}(u, v) =(\nabla u, \nabla v),
$$
where $\tau=k^2$. Using the Green's formula, the variational problem can be written as the following generalized eigenvalue problem: Find $(\tau, u)\in \mathbb{R}\times H_0^2(\Omega)$ such that
\begin{equation} \label{eq: variational_form_Helm}
	\mathcal{A}_\tau(u, v)  = \tau\mathcal{B}(u, v),
\end{equation}
where the bilinear form $\mathcal{A}_\tau(\cdot,\cdot)$ is coercivity on $H_0^2(\Omega)\times H_0^2(\Omega)$ and the binear form $\mathcal{B}(\cdot,\cdot)$ is symmetric and nonnegative on $H_0^2(\Omega)\times H_0^2(\Omega)$ (c.f. \cite{sun2011iterative}).

\subsection{Subdivisions and finite elements}
Let $\big\{\mathcal{G}_h\big\}$ be a family of rectangular grids of domain $\Omega$. If none of the vertices of a cell is on $\partial\Omega$, it is called an interior cell, otherwise it is called a boundary cell. We use $\mathcal{K}_{h}^{i}$ and $\mathcal{K}_{h}^{b}$ for the set of interior cells and boundary cells, respectively.  Let $\overline{\omega}$ and $\mathring{\omega}$  denote the closure and the interior of a region $\omega$. We use symbol $\#$ for the cardinal number of a set. For an edge $e$, $\mathbf{n}_e$ is a unit vector normal to $e$ and $\mathbf{t}_e$ is a unit tangential vector of $e$ such that $\mathbf{n}_e\times \mathbf{t}_e>0$. We use $\llbracket\cdot\rrbracket_e$ for the jump across $e$. We stipulate that, if $e = T_{1}\cap T_{2}$, then $\llbracket v \rrbracket_e = \big(v|_{T_{1}} - v|_{T_{2}}\big)|_{e}$ if the direction of $\mathbf{n}_{e}$ goes from $T_{1}$ to $T_{2}$, and if $e\subset\partial\Omega$, then $\llbracket\cdot\rrbracket_e$ is the evaluation on $e$. 

Suppose that $K$ represents a rectangle with sides parallel to the two axis respectively. Let $a_{i}$ and $e_{i}$ denote a vertex and an edge of $K$ with $i = 1:4$. Let $c_{K}:= (x_{K},y_{K})$ be the barycenter of $K$. Let $h_{x,K}$, $h_{y,K}$ be the length of $K$ in the $x$ and $y$ directions, respectively.  Let $h_{K}: = \max\{h_{x,K},h_{y,K}\}$ be the size of $K$, and $\rho_{K}$ be the inscribed circle radius.
Let $h := \max\limits_{K \in \mathcal{G}_{h}}h_{K}$ be the mesh size of $\mathcal{G}_{h}$. Let $P_l(K)$ denote the space of all polynomials on $K$ with the total degree no more than~$l$. Let $Q_l(K)$ denote  the space of all polynomials on $K$ of degree no more than~$l$ in each variable. Similarly, we define spaces $P_l(e)$ and $Q_l(e)$ on an edge~$e$.

In this paper, we assume that $\big\{\mathcal{G}_{h}\big\}$ is in a regular family of rectangular subdivisions, i.e.,
\begin{equation}\label{eq:regularity}
\max_{K\in \mathcal{G}_{h}}\frac{h_{K}}{\rho_{K}} \leq \gamma_{0},
\end{equation}
where $\gamma_{0}$ is a generic constant independent of $h$.  Such a mesh is actually locally quasi-uniform, and this helps for the stability analysis of the interpolation operator constructed in Section~\ref{sec:rrmscheme}.

The rectangular Morley (RM) element is defined by 
 $(K,P_{K}^{\rm{M}},D_{K}^{\rm{M}})$ with the following properties:
\begin{itemize}
\item[(1)] $K$ is a rectangle;
\item[(2)] $P_{K}^{\rm{M}} = P_{2}(K) + \text{span}\{x^{3},y^{3}\}$;
\item[(3)] for any $v\in H^{2}(K)$, $D_{K}^{\rm{M}} =\big\{ v(a_{i}), \ \fint_{e_{i}}\partial_{\mathbf{n}_{e_{i}}}v \, d s \big\}_{i=1:4}$.
\end{itemize}
Given a grid $\mathcal{G}_h$, define the RM element space on $\mathcal{G}_h$ as
\begin{equation*}
\begin{split}
V_{h}^{\rm{M}}(\mathcal{G}_h) := \Big\{w_{h}\in L^{2}(\Omega) : w_{h}|_{K} \in P_{K}^{\rm{M}}, \ & w_{h}(a)\mbox{ is continuous at any } a \in \mathcal{N}_{h}^{i}, \\
& \mbox{and} \fint_{e}\partial_{\mathbf{n}_{e}} w_{h} \, d s \mbox{ is continuous across any } e \in \mathcal{E}_{h}^{i} \Big\}.
\end{split}
\end{equation*}
Associated with the boundary condition of $H^{1}_{0}$ type, define $V_{hs}^{\rm{M}}(\mathcal{G}_h) :=\Big\{w_{h}\in V_{h}^{\rm{M}} :  w_{h}(a)=0, \ \forall a\in \mathcal{N}_{h}^{b}\Big\}$, and associated with the boundary condition of $H^{2}_{0}$ type, define $V_{h0}^{\rm{M}}(\mathcal{G}_h) :=\Big\{w_{h}\in V_{hs}^{\rm{M}} :\   \fint_{e}\partial_{\mathbf{n}_{e}} w_{h} \,d s =0,  \ \forall e\in \mathcal{E}_{h}^{b} \Big\}.$ In the sequel, we can drop the dependence on $\mathcal{G}_h$ when no ambiguity is brought in. 

For $v, w \in L^2(\Omega)$ that $\left.v\right|_K,\left.w\right|_K \in H^1(K), \forall K \in \mathcal{G}_h$, we define $ b_h(u, v)=\sum_{K \in \mathcal{G}_h} \int_K  \nabla u \cdot \nabla v  d x .$ 

\begin{lemma}{\rm(\!\cite[Lemmas 3.2 and 3.5]{XY.Meng;XQ.Yang;S.Zhang2016})}\label{lem:consisRM} For any function $v_{h} \in V_{hs}^{\rm{M}}$, we have the following estimates:
\begin{itemize}
\item[(a)] For any shape-regular rectangular grid, it holds that
\begin{align*}
| b_{h}(v,v_{h}) + (\Delta v,v_{h}) | \lesssim \sum_{K\in \mathcal{G}_{h}} h_{K}^{2} |v|_{2,K} |v_{h}|_{2,K}\lesssim h |v|_{2,\Omega} |v_{h}|_{1,h}, \quad \forall v\in H^{2}(\Omega)\cap H^{1}_{0}(\Omega);
\end{align*}

\item[(b)] For any uniform rectangular grid, it holds that
\begin{align*}
| b_{h}(v,v_{h}) + (\Delta v,v_{h}) | \lesssim h^{k-1} |v|_{k,\Omega} |v_{h}|_{1,h}, \quad \forall v\in H^{k}(\Omega)\cap H^{1}_{0}(\Omega), \quad k = 2,3.
\end{align*}
\end{itemize}
\end{lemma}

\section{An optimal interpolator to the reduced rectangular Morley element space}
\label{sec:rrmscheme}

\subsection{Reduced rectangular Morley element space revisited}
\label{sec:rrmrv}
Given a subdivision $\mathcal{G}_h$, the reduced rectangular Morley (RRM) element space \cite{Shuo.Zhang2020, ZZZ2021} thereon is defined as  
\begin{equation}
\begin{split}
V_{h}^{\rm{R}}(\mathcal{G}_h) := \Big\{w_{h}\in V^{\rm M}_h : w_{h}|_{K} \in P_{2}(K)\Big\}.
\end{split}
\end{equation}
The grid $\mathcal{G}_h$ may be omitted when there is no ambiguity induced. Associated with $H^{1}_{0}(\Omega)$, define $V_{hs}^{\rm{R}} :=V^{\rm R}_{h}\cap V^{\rm M}_{hs}$, and  associated with $H^{2}_{0}(\Omega)$, define $V_{h0}^{\rm{R}} :=V^{\rm R}_h \cap V^{\rm M}_{h0}$.

Denote, by $\mathcal{M}_{K}$,  a $3\times 3$ patch centered at $K$, with lengths and heights denoted by $\big\{L_{K,-1},\ L_{K},\ L_{K,1}\big\}$ and $\big\{H_{K,-1},\ H_{K},\ H_{K,1}\big\}$, respectively (see Figure~\ref{fig:3x3basis}). 
 
Let $\big\{X_{m,n}^{K}\big\}$, $\big\{Y_{m,n}^{K}\big\}$, and $\big\{Z_{m,n}^{K}\big\}$ denote the  interior vertices,  interior edge midpoints in the $x$ direction, and interior edge midpoints in the $y$ direction inside $\mathcal{M}_{K}$, respectively (see Figure~\ref{fig:3x3basis}). 

\begin{figure}[!htbp]
\centering
\includegraphics[height=0.34\hsize]{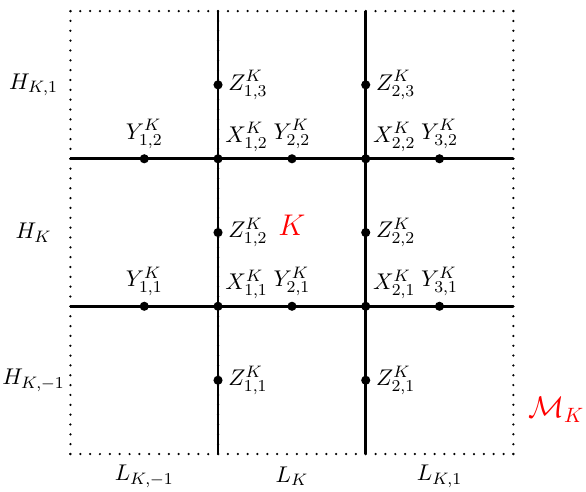}
\caption{Illustration of a $3 \times 3$ patch $\mathcal{M}_{K}$.}\label{fig:3x3basis}
\end{figure}
%

\begin{lemma}\rm(\!\cite[Lemma 15]{Shuo.Zhang2020})
Let $\mathcal{M}_{K}$ be a $3\times 3$ patch centered at $K$; see Figure~\ref{fig:3x3basis}.  Denote $V_{K}^{\rm{R}}:=V^{\rm R}_{h0}(\mathcal{M}_{K})$. Then ${\rm dim}(V_{K}^{\rm{R}} ) = 1$.
\end{lemma}

Now we give a detailed description of the functions in $V_{K}^{\rm{R}}$. 
For any $\varphi \in V_{K}^{R}$, we denote $v_{m,n}^{K}:=\varphi(X_{m,n}^{K})$, $u_{m,n}^{K}:=\partial_y\varphi(Y_{m,n}^{K})$, and $z_{m,n}^{K}:=\partial_x\varphi(Z_{m,n}^{K})$. Then the values of $\{v_{m,n}^{K}\}$,  $\{u_{m,n}^{K}\}$, and  $\{z_{m,n}^{K}\}$ satisfy that
\begin{align}
\big[v_{1,1}^{K},\ v_{2,1}^{K},\ v_{1,2}^{K},\ v_{2,2}^{K}\big] & = \big[1,  \gamma_{x}^{K},\ \gamma_{y}^{K},\ \gamma_{x}^{K}\gamma_{y}^{K}\big]\, v_{1,1}^{K}; \label{eq:values1Basis} 
\\
\big[u_{1,1}^{K},\ u_{2,1}^{K},\ u_{3,1}^{K},\ u_{1,2}^{K},\ u_{2,2}^{K},\ u_{3,2}^{K}\big] & = \Big[\tfrac{1}{H_{K,-1}},\ \tfrac{1+\gamma_{x}^{K}}{H_{K,-1}},\ \tfrac{\gamma_{x}^{K}}{H_{K,-1}},\ \tfrac{-\gamma_{y}^{K}}{H_{K,1}},\ \tfrac{-(1+\gamma_{x}^{K})\gamma_{y}^{K}}{H_{K,1}},\ \tfrac{-\gamma_{x}^{K}\gamma_{y}^{K}}{H_{K,1}}\Big]\, v_{1,1}^{K};\label{eq:values2Basis}
\\
\big[z_{1,1}^{K},\ z_{2,1}^{K},\ z_{1,2}^{K},\ z_{2,2}^{K},\ z_{1,3}^{K},\ z_{2,3}^{K}\big] &= \Big[\tfrac{1}{L_{K,-1}},\ \tfrac{-\gamma_{x}^{K}}{L_{K,1}},\ \tfrac{1+\gamma_{y}^{K}}{L_{K,-1}},\ \tfrac{-(1+\gamma_{y}^{K})\gamma_{x}^{K}}{L_{K,1}},\ \tfrac{\gamma_{y}^{K}}{L_{K,-1}},\ \tfrac{-\gamma_{x}^{K}\gamma_{y}^{K}}{L_{K,1}}\Big]\, v_{1,1}^{K},\label{eq:values3Basis}
\end{align}
where $\gamma_{x}^{K}= \frac{1+\frac{L_{K}}{L_{K,-1}}}{1+\frac{L_{K}}{L_{K,1}}}$ and $\gamma_{y}^{K}= \frac{1+\frac{H_{K}}{H_{K,-1}}}{1+\frac{H_{K}}{H_{K,1}}}$. For each vertice $X_{m,n}^{K}$, midpoint $Y_{m,n}^{K}$, or midpoint $Z_{m,n}^{K}$ on the boundary of $\mathcal{M}_{K}$, $v_{m,n}^{K}$, $u_{m,n}^{K}$, or $z_{m,n}^{K}$ equals to zero correspondingly. Therefore, $\varphi \in V_{K}^{\rm{R}}$ is uniquely determined, once $\varphi(X_{1,1}^{K})$ is fixed. 
\begin{definition}\label{def:3x3basis}{
\rm
Let $\mathcal{M}_{K}$ be a $3\times 3$ patch with a center cell $K$.  Denote, by $\varphi_{K}$, a basis function supported on $\mathcal{M}_{K}$, which satisfies
\begin{itemize}
\item[(a)]$\varphi_{K}(x,y)\equiv 0,\quad \forall (x,y)\notin  \mathcal{M}_{K}$;
\item[(b)] $\varphi_{K}|_{\mathcal{M}_{K}} \in V_{K}^{\rm{R}}$, and specially $\varphi_{K}(X_{1,1}^{K}) = \frac{L_{K,-1}}{L_{K,-1}+L_{K}}\cdot \frac{H_{K,-1}}{H_{K,-1}+H_{K}}$.
	\end{itemize}	
}
\end{definition}

\begin{remark}
{\rm
The assumption of $\varphi_{K}(X_{1,1}^{K}) = \frac{L_{K,-1}}{L_{K,-1}+L_{K}}\cdot \frac{H_{K,-1}}{H_{K,-1}+H_{K}}$ in the definition is not necessary, but can facilitate the subsequent analysis of the properties of basis functions. 
}
\end{remark}

\begin{proposition}\label{prop:scaling}
Let $\varphi_{K}$ be a function given in Definition~\ref{def:3x3basis}. For $k \geqslant 0$, it holds that
\begin{equation}\label{eq:normEstimate}
\big|\varphi_{K}|_{k,T}\leq C_{\gamma_{0}}h_{T}^{1-k}, \quad \forall T \subset \mathcal{M}_{K}, 
\end{equation}
where $C_{\gamma_{0}}$ represents a positive constant only dependent on the regularity constant $\gamma_{0}$.
\end{proposition}
The proof is postponed to Section \ref{sec:pfprop:scaling}

Recall that $\mathcal{K}_{h}^{i}$ and $\mathcal{K}_{h}^{b}$ represent the set of interior cells and boundary cells, respectively. For any $K\in \mathcal{K}_{h}^{i}$, there exists a $3\times 3$ patch $\mathcal{M}_{K}$ which is within $\Omega$, i.e., its nine cells locate in $\mathcal{G}_{h}$.
\begin{lemma}\rm(\cite{Shuo.Zhang2020})
$\displaystyle V^{\rm R}_{h0}={\rm span}\{\varphi_K\}_{K\in\mathcal{K}_{h}^{i}}$.
\end{lemma}

\begin{proposition}\label{pro:property of phi 5x5}
Let $K$ be a cell in $ \mathcal{K}_{h}^{i}$, and $\mathcal{M}_{K}$ be its corresponding $3\times 3$ patch. Assume that the $5\times 5$ patch centered at $K$ is within $\Omega$, i.e., these $25$ elements locate in $\mathcal{G}_{h}$. Define $A_{K}:= \big\{K_{dl},\ K_{d},\ K_{dr},\ K_{l},\ K,\ K_{r},\ K_{ul},\ K_{u}, \ K_{ur} \big\};$ see Figure~\ref{fig:5x5basis}. Then $K$ is located in the supports of nine functions in $\big\{\varphi_{T}:  \ T\in A_{K}\big\}$, and
\begin{itemize}
\item[(a)] for any $v\in Q_{1}(\mathcal{M}_{K})$ and $(x,y)\in K$, it holds that
\begin{align*}
\sum_{T\in A_{K}} v(x_{T},y_{T})\,\varphi_{T}(x,y) = v(x,y) \quad \mbox{and}\quad 
\sum_{T\in A_{K}} (\fint_{T}v{\,d x d y })\, \varphi_{T}(x,y)= v(x,y),
\end{align*}
where $v(x_{T},y_{T})$ is the value of $v$ at the barycenter of $T$;
\item[(b)] for any $v \in P_{2}(\mathcal{M}_{K})$ and $(x,y)\in K$, it holds that
\begin{align*}
&\sum_{T\in A_{K}} r_{T}(v)\,\varphi_{T}(x,y) = v(x,y)\ \mbox{ with} \ \  r_{T}(v) = v(x_{T},y_{T}) - \tfrac{1}{8}\big(L_{T}^{2}\,\tfrac{\partial^{2}v}{\partial x^{2}} + H_{T}^{2}\,\tfrac{\partial^{2}v}{\partial y^{2}} \big),
\\
&\sum_{T\in A_{K}} t_{T}(v)\,\varphi_{T}(x,y) =v(x,y)\ \mbox{ with} \ \ t_{T}(v) = \fint_{T}v {\,d x d y }  - \tfrac{1}{6}\big(L_{T}^{2}\,\tfrac{\partial^{2}v}{\partial x^{2}}  + H_{T}^{2}\,\tfrac{\partial^{2}v}{\partial y^{2}} \big);
\end{align*}
\item[(c)] for any $(x,y)\in K$, there exists a set of coefficients $\{d_{T}\}$, such that 
\begin{align*}
\sum_{T\in A_{K}} (d_{T}L_{T}H_{T})\,\varphi_{T}(x,y) = 0, \quad \forall (x,y)\in K,
\end{align*}
where $d_{K} = \pm 1$, $d_{K_{d}} = d_{K_{u}} = d_{K_{l}} = d_{K_{r}} = - d_{K}$, and $d_{K_{dl}} = d_{K_{dr}} = d_{K_{ul}} = d_{K_{ur}}  = d_{K}$.
\end{itemize}
\end{proposition}
The proof is postponed to Section \ref{sec:pfpro:property of phi 5x5}.

\begin{figure}[!htbp]
\centering
\includegraphics[height=0.34\hsize]{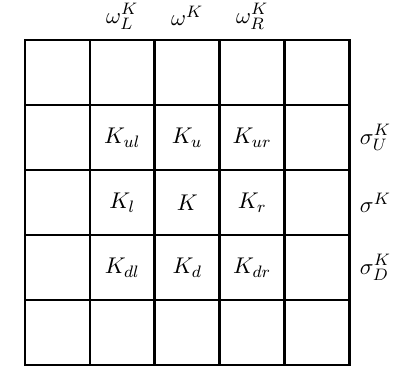}
\caption{Illustration of a $5 \times 5$ patch centered at $K$.}\label{fig:5x5basis}
\end{figure}

Notice that, for a single function $\varphi_{T}$, it is dependent on the lengths and heights of $\mathcal{M}_{T}$, i.e., $\big\{L_{K,-1},\ L_{K},\ L_{K,1}\big\}$ and $\big\{H_{K,-1},\ H_{K},\ H_{K,1}\big\}$. The following remark says that specific combinations of several functions, when restricted on specific cells, can be independent on some lengths or heights correspondly.
 As is shown in  Figure~\ref{fig:5x5basis}, we denote 
\begin{align*}
\omega_{L}^{K}:= K_{dl}\cup K_{l}\cup K_{ul}, \quad \omega^{K}:= K_{d}\cup K\cup K_{u}, \quad \omega_{R}^{K}:= K_{dr}\cup K_{r}\cup K_{ur}, \\
\sigma_{D}^{K} := K_{dl}\cup K_{d}\cup K_{dr}, \quad \sigma^{K} := K_{l}\cup K \cup K_{r}, \quad \sigma_{U}^{K} := K_{ul}\cup K_{u}\cup K_{ur}.
\end{align*}
 This result is obtained by direct calculation.

\begin{remark}\label{rem:adjoint basis property}
{\rm
Suppose that $K$ is located in the supports of nine functions $\big\{\varphi_{T}:  \ T\in A_{K}\big\}$.
For $v \in \{1,\ x, \ y,\ xy, \ x^{2}, \ y^{2}\}$, it holds that}
\begin{itemize}
\item[(a)]
{\rm The functions 
$\big(r_{K}(v)\,\varphi_{K} + r_{K_{r}}(v)\,\varphi_{K_{r}}\big)|_{\omega_{L}^{K}\cup \omega^{K}}$,  $\big(t_{K}(v)\,\varphi_{K} + t_{K_{r}}(v)\,\varphi_{K_{r}}\big)|_{\omega_{L}^{K}\cup \omega^{K}}$, and $\big(d_{K}L_{K}H_{K}\,\varphi_{K} + d_{K_{r}} L_{K_{r}}H_{K_{r}}\,\varphi_{K_{r}}\big)|_{\omega_{L}^{K}\cup \omega^{K}}$ are independent of $L_{K,1} \ \big( =L_{K_{r}}\big)$ and $L_{K_{r},1}.$ 
}
\item[(b)]
{\rm
The functions 
$\big(r_{K_{l}}(v)\,\varphi_{K_{l}} + r_{K}(v)\,\varphi_{K}\big)|_{\omega^{K} \cup \omega_{R}^{K}}$, $\big(t_{K_{l}}(v)\,\varphi_{K_{l}} + t_{K}(v)\,\varphi_{K}\big)|_{\omega^{K}\cup \omega_{R}^{K}}$ and $\big(d_{K_{l}}L_{K_{l}}H_{K_{l}}\,\varphi_{K_{l}} + d_{K}L_{K}H_{K}\,\varphi_{K}\big)_{\omega^{K}\cup \omega_{R}^{K}}$ are  independent of $L_{K_{l},-1}$ and $L_{K,-1} \ \big(= L_{K_{l}}\big).$
}
\item[(c)]
{\rm The functions 
$\big(r_{K}(v)\,\varphi_{K} + r_{K_{u}}(v)\,\varphi_{K_{u}}\big)|_{\sigma_{D}^{K}\cup \sigma^{K}}$, $\big(t_{K}(v)\,\varphi_{K} + t_{K_{u}}(v)\,\varphi_{K_{u}}\big)|_{\sigma_{D}^{K}\cup \sigma^{K}}$ and $\big(d_{K}L_{K}H_{K}\,\varphi_{K} +d_{K_{u}}L_{K_{u}}H_{K_{u}}\,\varphi_{K_{u}}\big)_{\sigma_{D}^{K}\cup \sigma^{K}}$ are independent of $H_{K,1} \ \big(= H_{K_{u}}\big)$ and $H_{K_{u},1}.$
}
\item[(d)]
{\rm The functions
$\big(r_{K_{d}}(v)\,\varphi_{K_{d}} + r_{K}(v)\,\varphi_{K}\big)|_{\sigma^{K}\cup \sigma_{U}^{K}}$, $\big(t_{K_{d}}(v)\,\varphi_{K_{d}} + t_{K}(v)\,\varphi_{K}\big)|_{\sigma^{K}\cup \sigma_{U}^{K}}$ and $\big(d_{K_{d}}L_{K_{d}}H_{K_{d}}\,\varphi_{K_{d}} + d_{K}L_{K}H_{K}\,\varphi_{K}\big)_{\sigma^{K}\cup \sigma_{U}^{K}}$ are independent of $H_{K_{d},-1}$ and $H_{K,-1}  \ \big(=H_{K_{d}}\big).$
}
\end{itemize}
\end{remark}

As a basic property of $V_{h0}^{\rm R}$, a main result of the paper is the lemma below. 
\begin{lemma}[discrete analogue of \eqref{eq:Grisvard}] \label{lem: discrete_property_1}For $u_h, v_h \in V_{h0}^{\rm R}$, 
\begin{equation}  
\label{eq: discrete strengthened Miranda-Talenti estimate}
\sum_{K \in \mathcal{G}_h} \int_K \Delta u_h \cdot \Delta v_h = \sum_{K \in \mathcal{G}_h} \int_K \nabla^2u_h : \nabla^2 v_h.
\end{equation}
\end{lemma}

\begin{proof}
Firstly, let's revisit the Rectangular-Morley element on the unit square $(\xi,\eta)\in[0,1]^2$. Denote $\phi_1,...,\phi_4$ be the nodal basis function associated with vertices and $\phi_5,...,\phi_8$ be the nodal basis function associated with edges.

Now we explicitly present the basis function $\varphi_K$ in definition \ref{def:3x3basis} supported on $\mathcal{M}_K$, the patch centered at $K$ cf.Figure \ref{fig:3x3basis}. Denote $\{\varphi^K_i\}_{i = 1}^9$ by the restriction of $\varphi_K$ on the i-th subrectangle. By the definition of RRM element and the change of variables from the unit square to each subrectangle, 
we can get 
\begin{equation} \label{eq: patch_basis}
	\varphi^K_{i} = \sum_{j = 1}^8\mathbb{COE}_{i,j}\widehat{\phi_j} \cdot v_{1,1}^K
\end{equation}
where $\widehat{\phi_j} = \phi_j(\xi(x,y),\eta(x,y))$ and 

{
\footnotesize
$$
\mathbb{C O E}=\left[\begin{array}{cccccccc}
0 & 0 & 1 & 0 & 0 & 1 & 1 & 0 \\
0 & 1 & \gamma_x^K & 0 & \frac{L_K}{L_{K,-1}} & 1+\gamma_x^K & -\frac{L_k}{L_{K, 1}} \gamma_x^K & 0 \\
0 & \gamma_x^K & 0 & 0 & -\gamma_x^K & \gamma_x^K & 0 & 0 \\
0 & 0 & \gamma_y^K & 1 & 0 & -\frac{H_K}{H_{K, 1}} \gamma_y^K & 1+\gamma_y^K & \frac{H_K}{H_{K,-1}} \\
1 & \gamma_y^K & \gamma_y^K \gamma_x^K & \gamma_x^K & \frac{L_K}{L_{K,-1}}\left(1+\gamma_y^K\right) & -\frac{H_K}{H_{K, 1}} \gamma_y^K\left(1+\gamma_x^K\right) & -\frac{L_K}{L_{K, 1}} \gamma_x^K\left(1+\gamma_y^K\right) & \frac{H_K}{H_{K_{,}-1}\left(1+\gamma_x^K\right)} \\
\gamma_x^K & \gamma_y^K \gamma_x^K & 0 & 0 & -\left(1+\gamma_y^K\right) \gamma_x^K & -\frac{H_K}{H_{K, 1}} \gamma_y^K \gamma_x^K & 0 & \frac{H_K}{H_{K,-1}} \gamma_x^K \\
0 & 0 & 0 & \gamma_y^K & 0 & 0 & \gamma_y^K & -\gamma_y^K \\
\gamma_y^K & 0 & 0 & \gamma_y^K \gamma_x^K & \frac{L_K}{L_{K,-1}} \gamma_y^K & 0 & -\frac{L_K}{L_{K, 1}} \gamma_y^K \gamma_x^K & -\left(1+\gamma_x^K\right) \gamma_y^K \\
\gamma_x^K \gamma_y^K & 0 & 0 & 0 & -\gamma_y^K \gamma_x^K & 0 & 0 & -\gamma_y^K \gamma_x^K
\end{array}\right]
$$
}
actually we can simplify \eqref{eq: patch_basis} by setting $v_{1,1}^K = 1$. 
	
Here we just show the most complicated situation: Fix a rectangle $K_1$,
there exists a $5\times 5$ patch centered at $K_1$ cf.Figure \ref{fig:Miranda-Talentiestimate}(Left) denoted by $F_{K_1}$, such that for any $K'\in F_{K_1}$, $\mathcal{M}_{K'}$ is contained in the mesh. Notice that only these basis function $\{\varphi_{K'}: K' \in F_{K_1}\}$ can have the overlapping supports with $\varphi_{K_1}$. For the sake of brevity, we just check \eqref{eq: discrete strengthened Miranda-Talenti estimate} for $\varphi_{K_1}$ and $\varphi_{K_2}$ in detail cf.Figure \ref{fig:Miranda-Talentiestimate}(Right):
\begin{figure}[!htbp]
\centering
\includegraphics[height=0.34\hsize]{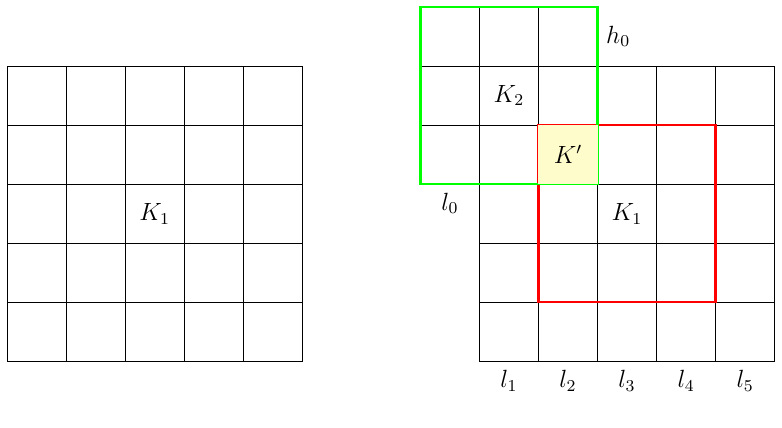}
\caption{Left: Illustration of $5 \times 5$ patch $F_{K_1}$. Right: Illustration of the basis function $\varphi_{K_1}$ and $\varphi_{K_2}$}\label{fig:Miranda-Talentiestimate}
\end{figure}

\begin{align*}
	\sum_{K \in \mathcal{G}_h} \int_K \Delta \varphi_{K_1} \cdot \Delta \varphi_{K_2} & = \int_{K'} \Delta \varphi_{K_1} \cdot \Delta \varphi_{K_2} \\
	& = \int_{K'}\left(\gamma_y^{K_1}\Delta\widehat{\phi_4} + \gamma_y^{K_1}\Delta\widehat{\phi_7}- \gamma_y^{K_1}\Delta\widehat{\phi_8}\right) \left(\gamma_x^{K_2}\Delta\widehat{\phi_2} - \gamma_x^{K_2}\Delta\widehat{\phi_5}+ \gamma_x^{K_2}\Delta\widehat{\phi_6}\right) \\
	& = |K'| \int_{[0,1]^2} \left[\frac{1}{l_2^2}\gamma_y^{K_1} \left(\frac{\partial^2\phi_4}{\partial\xi^2} + \frac{\partial^2\phi_7}{\partial\xi^2} - \frac{\partial^2\phi_8}{\partial\xi^2}\right) +  \frac{1}{h_2^2}\gamma_y^{K_1}\left(\frac{\partial^2\phi_4}{\partial\eta^2} + \frac{\partial^2\phi_7}{\partial\eta^2} - \frac{\partial^2\phi_8}{\partial\eta^2} \right)\right]\\
	& \cdot \left[\frac{1}{l_2^2}\gamma_x^{K_2} \left(\frac{\partial^2\phi_2}{\partial\xi^2} - \frac{\partial^2\phi_5}{\partial\xi^2} + \frac{\partial^2\phi_6}{\partial\xi^2}\right) +  \frac{1}{h_2^2}\gamma_x^{K_2}\left(\frac{\partial^2\phi_2}{\partial\eta^2} - \frac{\partial^2\phi_5}{\partial\eta^2} + \frac{\partial^2\phi_6}{\partial\eta^2} \right)\right]\\
	& = |K'| \int_{[0,1]^2} \left[ \frac{1}{l_2^2}\gamma_y^{K_1}\left( 3 - 6\xi + 6\xi - 2 \right) + \frac{1}{h_2^2}\gamma_y^{K_1}\left( 6\eta - 3 - (6\eta - 4)\right)\right]\\
	& \cdot \left[ \frac{1}{l_2^2}\gamma_x^{K_2}\left( 6\xi - 3 - (6\xi - 4) \right) + \frac{1}{h_2^2}\gamma_x^{K_2}\left( 3 - 6\eta + 6\eta - 2\right)\right]\\
	& = |K'| \int_{[0,1]^2} \left(\frac{1}{l_2^2}\gamma_y^{K_1} + \frac{1}{h_2^2}\gamma_y^{K_1}\right) \cdot \left(\frac{1}{l_2^2}\gamma_x^{K_2} +  \frac{1}{h_2^2}\gamma_x^{K_2}\right).
\end{align*}

\begin{align*}
	\sum_{K \in \mathcal{G}_h} \int_K \nabla^2\varphi_{K_1} : \nabla^2 \varphi_{K_2} & = \int_{K'} 	\nabla^2\varphi_{K_1} : \nabla^2 \varphi_{K_2} \\
	& = |K'| \int_{[0,1]^2} \gamma_y^{K_1}\nabla^2(\widehat{\phi_4} + \widehat{\phi_7}- \widehat{\phi_8}) : \gamma_x^{K_2}\nabla^2(\widehat{\phi_2} - \widehat{\phi_5} + \widehat{\phi_6})\\
	& = |K'| \int_{[0,1]^2} \gamma_y^{K_1}\gamma_x^{K_2} \left[\begin{array}{cc}
	1/l_2^2 , &-1/(l_2h_2)\\
	-1/(l_2h_2),&1/h_2^2	
\end{array}\right] : \left[\begin{array}{cc}
	1/l_2^2 , &-1/(l_2h_2)\\
	-1/(l_2h_2),&1/h_2^2	
\end{array}\right]\\
& = \sum_{K \in \mathcal{G}_h} \int_K \Delta \varphi_{K_1} \cdot \Delta \varphi_{K_2}.
\end{align*}
Similarly, \eqref{eq: discrete strengthened Miranda-Talenti estimate} holds for other basis functions by direct computation.

\end{proof}

\subsection{An auxiliary interpolator on an extended grid}
\label{sec:aied}

\subsubsection{Extension of the grid} 
We introduce here a virtual extension of the grids. With a quick glance to Figure \ref{figure:expansionmesh} illustrating the whole extension later, we firstly introduce in detail the rule how the grid is extended.

A boundary vertex in $\mathcal{N}_{h}^{b}$ is called a {\it corner node} if it is an intersection of two boundaries of $\partial \Omega$, which are not on the same line. It can be divided into {\it convex corner node} or {\it concave corner node}. It is assumed that any two corner nodes are not contained in the same cell.   

A boundary edge in $\mathcal{E}_{h}^{b}$ is called a {\it corner edge} if one of its endpoints is a corner node, otherwise it is named as a {\it non-corner boundary edge}.

\begin{itemize}
\item Consider a convex corner as shown in Figure~\ref{fig:expansion at a corner} (left). Let $L_{K,1}$, $L_{K,2}$, $H_{K,1}$, and $H_{K,2}$ be some constants close to the size of $K$.  Complete a $3\times 3$ patch, denoted by $\mathcal{M}_{K}$, outside the domain with $K$ as the center. The element to the right of $K$ is denoted as $K_{r}$, the element above it is denoted as $K_{u}$, and the element opposite to $K$ with respect to the corner node is denoted as $K_{ur}$. Adding a layer of rectangles outside $\mathcal{M}_{K}$, we obtain four patches $\mathcal{M}_{K}$, $\mathcal{M}_{K_{r}}$, $\mathcal{M}_{K_{u}}$, and $\mathcal{M}_{K_{ur}}$ associated with this convex corner. And four functions supported on them are denoted as $\varphi_{K}$, $\varphi_{K_{r}}$, $\varphi_{K_{u}}$, and $\varphi_{K_{ur}}$, respectively. 

\item Consider a concave corner as shown in Figure~\ref{fig:expansion at a corner} (right). We also extend the mesh to get four $3\times 3$ patches, each of which is centered at $K$, $K_{r}$, $K_{u}$, and an added element $K_{ur}$, and we derive four functions supported on four $3\times 3$ patches correspondingly. 

\item Consider a non-corner boundary edge shown in Figure~\ref{fig:expansion at a non-corner edge} (left). Let $H_{K,1}$ and $H_{K,2}$ be two arbitrary constants close to the height of $K$. A $3\times 3$ patch $\mathcal{M}_{K}$, is completed outside the domain centered at $K$. The element opposite to $K$ with respect to the non-corner boundary edge is denoted as $K_{u}$. Extending a layer of rectangles outside $\mathcal{M}_{K}$, a $3\times 3$ patch centered at $K_{u}$ is derived and denoted as $\mathcal{M}_{K_{u}}$.  
Let $\varphi_{K}$ and $\varphi_{K_{u}}$ denote two functions supported on $\mathcal{M}_{K}$ and $\mathcal{M}_{K_{u}}$, respectively. Similar operations are conducted on the non-corner boundary edge in the vertical direction; see Figure~\ref{fig:expansion at a non-corner edge} (right). 
\end{itemize}
\begin{figure}[!htbp]
\centering
\includegraphics[height=0.33\hsize]{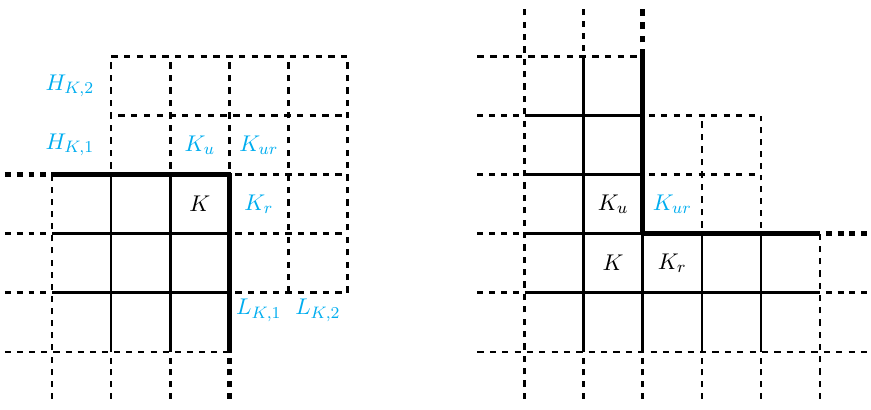}
\caption{Expansion outside $\Omega$ and four basis functions added corresponding to a convex corner node (left) and a concave corner node (right).}\label{fig:expansion at a corner}
\end{figure}
\begin{figure}[!htbp]
\centering
\includegraphics[height=0.33\hsize]{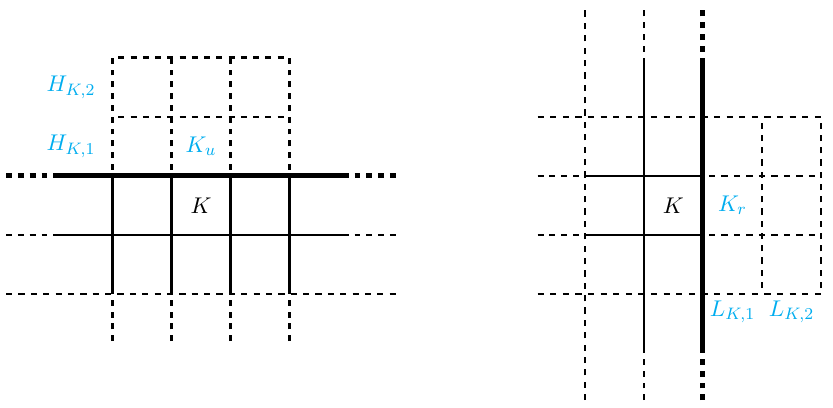}
\caption{Expansion outside a non-corner boundary edge (left or right) and two basis functions added.}\label{fig:expansion at a non-corner edge}
\end{figure}

The above expanding operations are carried out locally, by which each cell in $\mathcal{G}_{h}$ can be located in the supports of nine functions. For each boundary cell $K$, the choice of $L_{K,1}$, $L_{K,2}$, $H_{K,1}$, and $H_{K,2}$ appeared in Figures~\ref{fig:expansion at a corner} and \ref{fig:expansion at a non-corner edge}
can be determined only by the size of $K$, to ensure the regularity~\eqref{eq:regularity}.

Let $\mathcal{K}_{h}^{\rm ex}$ be the set of all newly added cells near corner nodes and non-corner boundary edges, such as $K_{u}$, $K_{r}$, $K_{ur}$ in Figures~\ref{fig:expansion at a corner} and \ref{fig:expansion at a non-corner edge}. Let $\mathcal{B}_{h} := \mathcal{K}_{h}^{b} \cup \mathcal{K}_{h}^{\rm ex}$. Then $\big\{ \mathcal{M}_{K} \big\}_{K \in \mathcal{B}_{h}}$ consists of patches which are not completely within  $\Omega$.  Denote $\mathcal{J}_{h}: = \mathcal{K}_{h}^{i} \cup \mathcal{B}_{h}$. Denote 
\begin{equation}
\widetilde{\Omega}_{h} = \cup_{K\in \mathcal{J}_{h}}\mathcal{M}_{K}\quad\mbox{(see Figure~\ref{figure:expansionmesh} (right).) }
\end{equation}
Then $\widetilde{\Omega}_{h}$ is a virtual expansion of the grid $\mathcal{G}_h$.

\begin{figure}[!htbp]
\centering
\includegraphics[height=0.48\hsize]{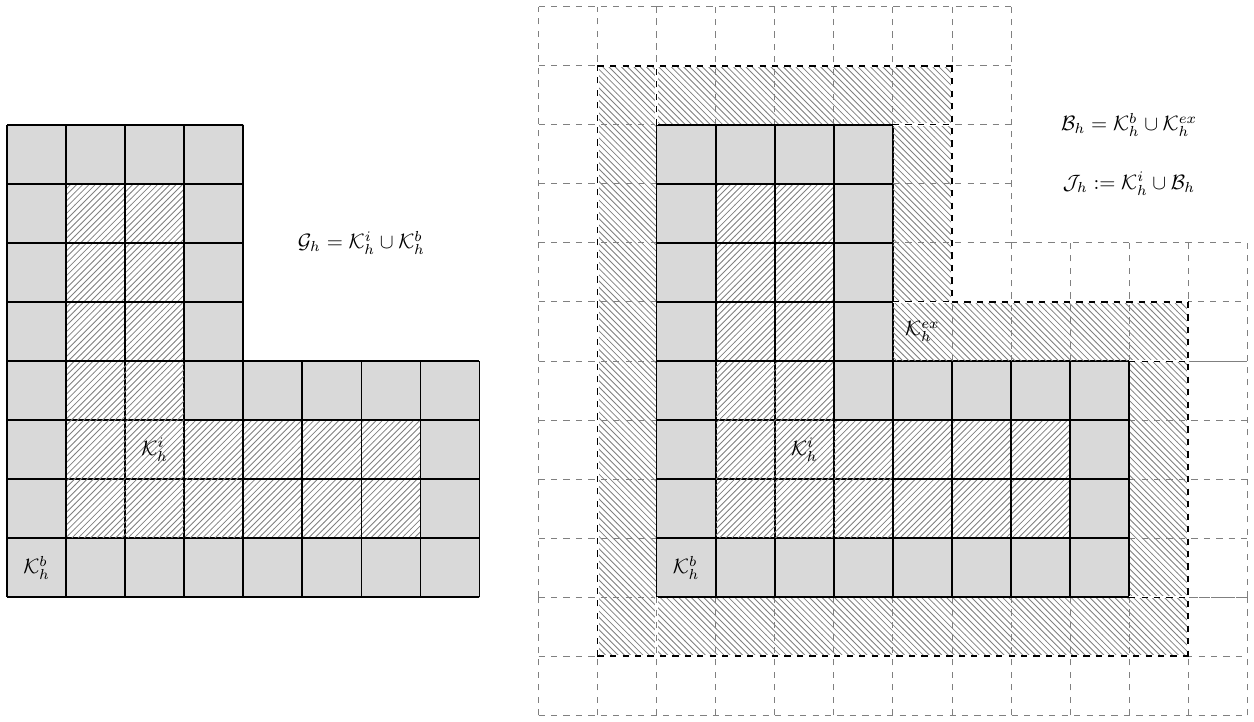}
\caption{Illustration of the extension of a grid. We use different kinds of shadow to show different parts of the grid. Left: an original mesh $\mathcal{G}_h$ on $\Omega$. Right: an extended grid the domain underlies which is denoted by $\widetilde{\Omega}_h$.} \label{figure:expansionmesh}
\end{figure}

\begin{lemma}\label{lem:property of phi}
The set of functions $\big\{\varphi_{K}\big\}_{K \in \mathcal{J}_{h}}$ have the following properties:
\begin{itemize}
\item[(a)] for any $v\in Q_{1}(\widetilde{\Omega}_{h})$ and $(x,y)\in \Omega$, it holds that
$$
\sum_{K\in \mathcal{J}_{h}}v(x_{K},y_{K})\, \varphi_{K}(x,y) = v(x,y) \quad \mbox{and} \quad \sum_{K\in \mathcal{J}_{h}}(\fint_{K} v{\,d x d y }) \,  \varphi_{K} = v(x,y),
$$
where $v(x_{K},y_{K})$ is the value of $v$ at the barycenter of $K$;
\item[(b)] for any $v \in P_{2}(\widetilde{\Omega}_{h})$ and $(x,y)\in \Omega$, it holds that
\begin{align*}
&\sum\limits_{K\in \mathcal{J}_{h}}r_{K}(v) \, \varphi_{K}(x,y) = v(x,y) \ \mbox{ with } \  r_{K}(v) = v(x_{K},y_{K}) - \tfrac{1}{8}\big(L_{K}^{2} \, \tfrac{\partial^{2}v}{\partial x^{2}} + H_{K}^{2}\, \tfrac{\partial^{2}v}{\partial y^{2}} \big),
\\
& \sum\limits_{K\in \mathcal{J}_{h}}t_{K}(v)\, \varphi_{K}(x,y)=v(x,y) \ \mbox{ with } \   t_{K}(v) = \fint_{K}v{\,d x d y}  - \tfrac{1}{6}\big(L_{K}^{2}\, \tfrac{\partial^{2}v}{\partial x^{2}}  + H_{K}^{2}\, \tfrac{\partial^{2}v}{\partial y^{2}} \big);
\end{align*}
\item[(c)]
for $(x,y)\in \Omega$, it holds that 
\begin{align*}
\sum\limits_{K\in \mathcal{J}_{h}} (d_{K} L_{K}H_{K})\, \varphi_{K}(x,y) = 0, \quad \forall (x,y)\in \Omega,
\end{align*}
where $\{d_{K}\}_{K\in\mathcal{J}_{h}}$ is a checkerboard coefficients set such that for any $K\in\mathcal{J}_{h}$, the following two conditions are satisfied:
(i) $d_{K} = \pm 1$, (ii) $d_{K_{l}} = d_{K_{r}} = d_{K_{d}} = d_{K_{u}} = d_{K}$; see Figure~\ref{fig:checkerboard}.
\end{itemize}
\end{lemma}
%
\begin{figure}[!htbp]
\centering
\includegraphics[height=0.32\hsize]{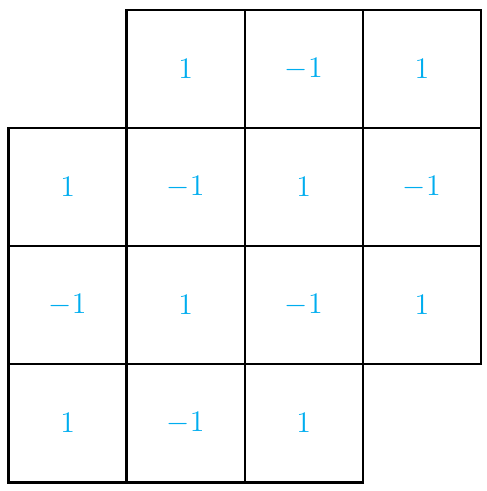}
\caption{Illustration of checkerboard distribution patterns.}\label{fig:checkerboard}
\end{figure}
\begin{proof}
It is equivalent to prove these equalities for each cell in $\mathcal{G}_{h}$. By Proposition~\ref{pro:property of phi 5x5}, for cells whose associated $5\times 5$ patches are within $\Omega$, these properties are already verified . Hence we only have to verify these equalities for the outermost two layers of cells in $\mathcal{G}_{h}$. 

Notice that the expanding operations are carried out locally, and each rectangle in $\mathcal{G}_{h}$ is located in the supports of nine functions  $\big\{\varphi_{K}\big\}_{K\in \mathcal{J}_{h}}$. Take a right boundary as an example; see Figure~\ref{fig:expansion at a non-corner edge} (right). According to Remark~\ref{rem:adjoint basis property}, the choices of $L_{K,1}$ and $L_{K,2}$ do not affect the values of $\big(r_{K}(v)\varphi_{K} + r_{K_{r}}(v)\varphi_{K_{r}}\big)|_{\Omega}$,  $\big(t_{K}(v)\varphi_{K} + t_{K_{r}}(v)\varphi_{K_{r}}\big)|_{\Omega}$, and $\big(d_{K}L_{K}H_{K}\varphi_{K} + d_{K_{r}}L_{K_{r}}H_{K_{r}}\varphi_{K_{r}}\big)|_{\Omega}$. Therefore, although these boundary elements on the same column may be extended outside $\Omega$ with different lengths, properties (a)-(c) stated in Property~\ref{pro:property of phi 5x5} are also true for elements 
located in the right outermost two layers of $\mathcal{G}_{h}$. 

The case of other boundaries can be verified similarly. These facts, together with Proposition~\ref{pro:property of phi 5x5}, complete the proof.
\end{proof}

\subsubsection{An auxiliary interpolator}

For a cell $K$ with $\big\{S_{\mu}^{K}\}_{\mu = 1:5}$ are five cells around $K$ (see Figure~\ref{fig:S3x3basis}), denote 
\begin{equation}
\lambda_{K}(v) = \sum\limits_{\mu = 1}^{5}w_{\mu}^{K}\,(\fint_{S_{\mu}^{K}}v{\,d x d y }),
\end{equation}  where
\begin{align*}
&w_{1}^{K} = \frac{-L_{K}^{2}}{(L_{K,-1}+L_{K})(L_{K,-1}+L_{K}+L_{K,1})}, \quad w_{2}^{K} = \frac{-L_{K}^{2}}{(L_{K}+L_{K,1})(L_{K,-1}+L_{K}+L_{K,1})},\\
&w_{3}^{K} = \frac{-H_{K}^{2}}{(H_{K,-1}+H_{K})(H_{K,-1}+H_{K}+H_{K,1})}, \quad w_{4}^{K} = \frac{-H_{K}^{2}}{(H_{K}+H_{K,1})(H_{K,-1}+H_{K}+H_{K,1})},\\
& w_{5}^{K} = 1 - (w_{1}^{K} + w_{2}^{K}  + w_{3}^{K}  + w_{4}^{K}).
\end{align*}

\begin{figure}[!htbp]
\centering
\includegraphics[height=0.34\hsize]{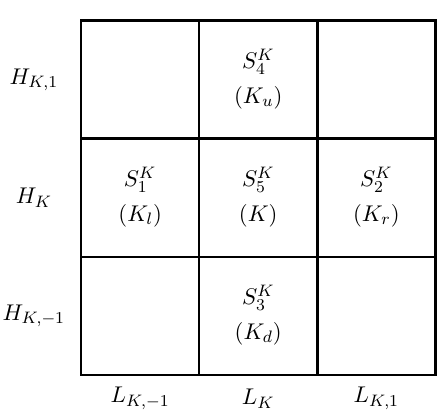}
\caption{Each functional $\lambda_{K}(v)$ utilizes the integral mean values of $v$ on the five elements around $K$.}\label{fig:S3x3basis}
\end{figure}

\begin{definition}\label{def:interpolationVh}
Define
\begin{align*}
\widetilde{\Pi}_{h}: L^{1}(\widetilde{\Omega}_{h}) \rightarrow  {  V_{h0}^{\rm{R}}(\widetilde{\Omega}_h)}, \quad\mbox{such\ that}\quad \widetilde{\Pi}_{h}(v) = \sum\limits_{K\in \mathcal{J}_{h}}\lambda_{K}(v)\,\varphi_{K}(x,y).
\end{align*}
\end{definition}
Denote 
\begin{equation}
\Delta_{T}  := \bigcup_{ \substack{ \mathcal{M}_{K}\supset T \\ K \in \mathcal{J}_{h}} } \mathcal{M}_{K},\ \ \ \mbox{for}\ \ T\in \mathcal{G}_{h}. 
\end{equation}
Evidently, 
$$
\mbox{for\ any}\ T\in \mathcal{G}_h,\  \widetilde{\Pi}_hv_1=\widetilde{\Pi}_hv_2\ \mbox{on}\ T\ \mbox{provided\ that}\ v_1=v_2 \ \mbox{on}\ \Delta_T.
$$
The fundamental properties of $\widetilde{\Pi}_h$, namely the local preservation of quadratic polynomials and the local stability, are surveyed in the Lemma below. 

\begin{lemma}  \label{lemma:ccc}
For any $T\in\mathcal{G}_h$, it holds that 
\begin{enumerate}
\item[(a)] $(\widetilde{\Pi}_{h}v)|_{T} = v|_{T}$ with $v\in L^1(\widetilde{\Omega}_h)$ such that $v|_{\Delta_T}\in P_{2}(\Delta_T)$;  \label{lem:preserveP2}
\item[(b)]$|\widetilde{\Pi}_{h}v|_{k,T} \lesssim h_{T}^{-k}\|v\|_{0,\Delta_{T}} \ \mbox{ with } \ 0 \leqslant k\leqslant 2$;
\label{lem:stability}
\item[(c)]$|v-\widetilde{\Pi}_{h}v|_{k,T}\lesssim h_{T}^{s-k}|v|_{s,\Delta_{T}},\ \mbox{for}\ v\in H^s(\Delta_T)\ \mbox{ with } 0\leq k \leq s \leq 3.$  \label{lem:approximationPi}
\end{enumerate}
\end{lemma}
We postpone the proof of Lemma \ref{lemma:ccc} to Section \ref{sec:pflm5}. With this lemma, we establish an available interpolation operator that is stable and reproduces quadratic polynomial.  Its construction is similar to the quasi-interpolation operators proposed in the spline theory \cite{Wang;Lu1998,Sablonniere.P2003,2Sablonniere.P2003}. As a matter of fact, an interpolation which does not necessarily preserve the entire finite element space but preserves quadratic polynomials locally admits the approximation property.

\subsection{An optimal interpolator to reduced rectangular Morley element space}
\label{sec:rrmitp}

\begin{definition}\label{def:interpolationVh0}
Define
\begin{align*}
\Pi_{h0}: L^{1}(\Omega) \rightarrow V_{h0}^{\rm{R}},\quad \Pi_{h0}v = \sum\limits_{K\in \mathcal{K}_{h}^{i}}\lambda_{K}(v)\,\varphi_{K}(x,y).
\end{align*}
\end{definition}
\begin{remark}
Since every functional $\lambda_{K}(v)$ in Definitions~ \ref{def:interpolationVh} and \ref{def:interpolationVh0} only involves the information of $v$ within $\mathcal{M}_{K}$, operators $\widetilde{\Pi}_{h}$ and $\Pi_{h0}$ define local approximation schemes. Notice that $\widetilde{\Pi}_{h}(v)$ involves information of $v$ outside $\Omega$, and the difference between $(\widetilde{\Pi}_{h}v)|_{\Omega}$ and $(\Pi_{h0}v)|_{\Omega}$ only lies in some cells near $\partial\Omega$. 
\end{remark}

\begin{remark}
We note that the operator $\Pi_{h0}$ is not a projection, i.e., it does not preserve every function in  $V_{h0}^{\rm R}$. Actually, with the given basis functions, no locally-defined interpolation can be projective; see \cite{ZZZ-AML} for details.
\end{remark}

\begin{lemma}
If $T\in \mathcal{G}_{h}$ and $\#\big\{\mathcal{M}_{K}:\ \mathcal{M}_{K}\cap \mathring{T} \ne \varnothing, K\in \mathcal{K}_{h}^{i}\big\} = 9$, then $(\Pi_{h0}v)|_{T} = v|_{T}$ with $v\in L^1(\Omega)$ such that $v|_{\Delta_T}\in P_2(\Delta_T)$.
\end{lemma}
\begin{proof}
The condition of $\#\big\{\mathcal{M}_{K}:\ \mathcal{M}_{K}\cap \mathring{T} \ne \varnothing, K\in \mathcal{K}_{h}^{i}\big\} = 9$ ensures that  $V_{h0}^{\rm R}\big|_{T} = V_{h}^{\rm R}\big|_{T}$, and the result is direct obtained from the  proof in Lemma \ref{lem:preserveP2}(b).
\end{proof}

A main result of the present paper is the theorem below. Note that herein $\Omega$ is not necessarily a convex domain.
\begin{theorem}\label{thm:approxH02}
It holds for $v\in H^{2}_{0}(\Omega) \cap H^{s}(\Omega)$ that 
\begin{equation}
|v-\Pi_{h0}v|_{k,h}\lesssim h^{s-k}|v|_{s,\Omega}, \quad 0\leq k \leq s, \quad 2 \leq s \leq 3.
\end{equation} 
\end{theorem}
We postpone the proof after some technical preparations.

These two lemmas are elementary but useful for verifying the approximation property of $\Pi_{h0}$; see, for example, \cite[Lemma~2]{Clement1975} and \cite[p24--p26]{Fichera1965}.
\begin{lemma}\label{lem:NormEquivalence} 
Let $e$ be an edge and $p\in P_{l}(e)$ with $l\geqslant 0$. Then 
$
|p|_{0,\infty, e}^{2} \lesssim |e|^{-1}|p|_{0,e}^{2}.
$
\end{lemma}
\begin{lemma}\label{lem:traceThm}
Let $K\in \mathcal{G}_{h}$, $e$ be an edge of $K$, and $v\in H^{1}(K)$. Then
$
|v|_{0,e}^{2}\lesssim h_{K}^{-1}|v|_{0,K}^{2} + h_{K}|v|_{1,K}^{2}.
$
\end{lemma}

\begin{lemma}{\rm(\!\cite[Theorem 1.4.5]{Brenner;Scott2007})}\label{lem:extension}
Suppose that $\Omega$ has a Lipschitz boundary. Then there is an extension mapping $E:~W^{p}_{k}(\Omega) \mapsto W^{p}_{k}(\mathbb{R}^{2})$ defined for all non-negative integers $k$ and real numbers $p$ in the range $1\leq p\leq \infty$ satisfying
 \begin{align}\label{eq:stabel extence}
 Ev\big|_{\Omega} = v, \quad 
 \|Ev\|_{W^{p}_{k}(\mathbb{R}^{2})} \leq C \|v\|_{W^{p}_{k}(\Omega)}, \quad \forall v \in W^{p}_{k}(\Omega),
 \end{align}
 where $C$ is a generic constant independent of $v$.
 \end{lemma}

\paragraph{\bf Proof of Theorem \ref{thm:approxH02}}
Let $E$ be an extension operator satisfying~\eqref{eq:stabel extence}. It holds by Lemma~\ref{lem:approximationPi}(c) and Lemma~\ref{lem:extension} for $v\in H^{s}(\Omega)$ that 
\begin{equation}\label{eq:approxH02}
|v-~\widetilde{\Pi}_{h}Ev|_{k,h} = |Ev-~\widetilde{\Pi}_{h}Ev|_{k,h} \lesssim h^{s-k}|Ev|_{s,\widetilde{\Omega}_{h}} \lesssim h^{s-k}|v|_{s,\Omega},\ \ \ 0\leq k \leq s \leq 3.
\end{equation}
Since $v-\Pi_{h0}v = (v - \widetilde{\Pi}_{h}Ev)+ (\widetilde{\Pi}_{h}Ev - \Pi_{h0}v)$, we only have to analyze $\widetilde{\Pi}_{h}Ev - \Pi_{h0}v$ cell by cell.
If $T\in \mathcal{G}_{h}$ and $\#\big\{\mathcal{M}_{K}: \ \mathcal{M}_{K}\cap \mathring{T} \ne \varnothing, K\in \mathcal{K}_{h}^{i}\big\} = 9$, then $(\widetilde{\Pi}_{h}Ev - \Pi_{h0}v)|_{T} = 0$, otherwise we have 
\begin{align}\label{eq:diff PiEv and Pi0v}
(\widetilde{\Pi}_{h}Ev - \Pi_{h0}v)|_{T} = \sum\limits_{\substack{\mathcal{M}_{K}\supset T \\ K \in \mathcal{J}_{h} \backslash \mathcal{K}_{h}^{i}}} \lambda_{K}(Ev) \varphi_{K}|_{T}.
\end{align}
First, we consider the case that $v\in H^{2}_{0}(\Omega)\cap H^{3}(\Omega)$. For any $p\in P_{2}(\widetilde{\Omega}_{h})$, by \eqref{eq:normEstimate} and the proof procedure in Lemma \ref{lem:stability}(b), we obtain
\begin{equation}\label{eq:lam(EV) phi}
\begin{split}
\big|\sum\limits_{\substack{\mathcal{M}_{K}\supset T \\ K \in \mathcal{J}_{h} \backslash \mathcal{K}_{h}^{i}}} \lambda_{K}(Ev) \varphi_{K} \big|_{k,T}^{2} & = \big|\sum\limits_{\substack{\mathcal{M}_{K}\supset T \\ K \in \mathcal{J}_{h} \backslash \mathcal{K}_{h}^{i}}} \lambda_{K}(Ev-p) \varphi_{K} + \sum\limits_{\substack{\mathcal{M}_{K}\supset T \\ K \in \mathcal{J}_{h} \backslash \mathcal{K}_{h}^{i}}}\lambda_{K}(p) \varphi_{K}\ \big|_{k,T}^{2} \\
& \lesssim \big|\sum\limits_{\substack{\mathcal{M}_{K}\supset T \\ K \in \mathcal{J}_{h} \backslash \mathcal{K}_{h}^{i}}} \lambda_{K}(Ev-p) \varphi_{K} \  \big|_{k,T}^{2} + \big|\sum\limits_{\substack{\mathcal{M}_{K}\supset T \\ K \in \mathcal{J}_{h} \backslash \mathcal{K}_{h}^{i}}}\lambda_{K}(p) \varphi_{K}\  \big|_{k,T}^{2} \\
& \lesssim h_{T}^{-2k}\|Ev-p\|_{0,\Delta_{T}}^{2} + h_{T}^{2-2k}\sum\limits_{\substack{\mathcal{M}_{K}\supset T \\ K \in \mathcal{J}_{h} \backslash \mathcal{K}_{h}^{i}}} |\lambda_{K}(p) |^{2}.
\end{split}
\end{equation}

From Lemma~\ref{lem:property of phi}(b) and the construction of the functional $\lambda_{K}$, it holds that 
$$
\lambda_{K}(p) = p(c_{K}) - \frac{1}{8}\frac{\partial^{2}p}{\partial x^{2}}L_{K}^{2} - \frac{1}{8}\frac{\partial^{2}p}{\partial y^{2}}H_{K}^{2},
$$
where $c_{K}= (x_{K},y_{K})$. Thus, by the Taylor's expansion, there exists some $a_{K}\in \mathcal{N}_{h}^{b}$, $e_{K}\in \mathcal{E}_{h}^{b}$, and a boundary cell $Q_{K}\in \mathcal{K}_{h}^{b}$, satisfying $a_{K}\in e_{K} \subset Q_{K}$, such that
\begin{align}\label{eq:lambdaP}
\lambda_{K}(p) = p(a_{K}) + (-1)^{\delta_{1}} \frac{\partial p }{\partial x}(a_{K})\frac{L_{K}}{2}+(-1)^{\delta_{2}}\frac{\partial p }{\partial y}(a_{K})\frac{H_{K}}{2} + (-1)^{\delta_{1}+\delta_{2}}\frac{\partial^{2} p }{\partial x \partial y}\frac{L_{K}H_{K}}{4},
\end{align}
where  $\delta_{1}$ and $\delta_{2}$ equals to $\pm 1$, and their specific values are determined by the relative position of $c_{K}$ and $a_{K}$. Since $v \in H^{2}_{0}(\Omega)$, it can be deduced that 
\begin{align}\label{eq:uH02}
|v|_{0,e_{K}} =\big|\tfrac{\partial v}{\partial x}\big|_{0,e_{K}} = \big|\tfrac{\partial v}{\partial y}\big|_{0,e_{K}} = \big|\tfrac{\partial^{2}v}{\partial x \partial y}\big|_{0,e_{K}} = 0.
\end{align}
From Lemma \ref{lem:traceThm} and \eqref{eq:uH02},  we have 
\begin{equation}\label{eq:traceP}
\begin{split}
& |p|_{0,e_{K}}^{2} = |v-p|_{0,e_{K}}^{2} \lesssim h_{Q_{K}}^{-1} |v-p|_{0,Q_{K}}^{2} + h_{Q_{K}} |v-p|_{1,Q_{K}}^{2}; \\
& \big|\tfrac{\partial p}{\partial x}\big|_{0,e_{K}}^{2} = \big|\tfrac{\partial v}{\partial x} -\tfrac{\partial p}{\partial x}\big|_{0,e_{K}}^{2} \lesssim h_{Q_{K}}^{-1} |v-p|_{1,Q_{K}}^{2} + h_{Q_{K}}|v-p|_{2,Q_{K}}^{2}; \\
& \big|\tfrac{\partial p}{\partial y}\big|_{0,e_{K}}^{2} = \big|\tfrac{\partial v}{\partial y} -\tfrac{\partial p}{\partial y}\big|_{0,e_{K}}^{2} \lesssim h_{Q_{K}}^{-1} |v-p|_{1,Q_{K}}^{2} + h_{Q_{K}} |v-p|_{2,Q_{K}}^{2}; \\
& \big|\tfrac{\partial^{2} p}{\partial x \partial y}\big|_{0,e_{K}}^{2} = \big|\tfrac{\partial^{2} v}{\partial x \partial y} -\tfrac{\partial^{2} p}{\partial x \partial y}\big|_{0,e_{K}}^{2} \lesssim h_{Q_{K}}^{-1} |v-p|_{2,Q_{K}}^{2} + h_{Q_{K}}  |v-p|_{3,Q_{K}}^{2}.
\end{split}
\end{equation}
A combination of Lemma \ref{lem:NormEquivalence},  \eqref{eq:bestP2}, and \eqref{eq:traceP} leads to 
\begin{align}\label{eq:s=3}
h_{T}^{-2k}\|Ev-p\|_{0,\Delta_{T}}^{2} + h_{T}^{2-2k}\sum\limits_{\substack{\mathcal{M}_{K}\supset T \\ K \in \mathcal{J}_{h} \backslash \mathcal{K}_{h}^{i}}} |\lambda_{K}(p)|^{2} \lesssim h_{T}^{2(3-k)} |Ev|_{3,\Delta_{T}} \ \mbox{ with } \ 0\leqslant k \leqslant 3.
\end{align}
By~\eqref{eq:diff PiEv and Pi0v}, \eqref{eq:lam(EV) phi}, \eqref{eq:s=3} and \eqref{eq:approxH02}, we derive $|v-\Pi_{h0}v|_{k,h}\lesssim h^{s-k}|v|_{3,\Omega}$ with $0\leqslant k \leqslant 3$.

For the case of a lower regularity that $v \in H^{2}_{0}(\Omega)$, we assume $p\in P_{1}(\widetilde{\Omega}_{h})$, and then $\lambda_{K}(p) = p(c_{K}) = p(a_{K}) + (-1)^{\delta_{1}} \frac{\partial p }{\partial x}(a_{K})\frac{L_{K}}{2}+(-1)^{\delta_{2}}\frac{\partial p }{\partial y}(a_{K})\frac{H_{K}}{2}$. For the case that $v \in H^{1}_{0}(\Omega)$, we utilize some $p\in P_{0}(\widetilde{\Omega}_{h})$, and then $\lambda_{K}(p) = p(c_{K}) = p(a_{K})$. By repeating the above process, similar results can be obtained for those two cases. Finally we complete the proof.
 \qed

\section{Reduced rectangular Morley element schemes for model problems}
\label{sec:convergence}

In this section, we establish the RRM element schemes for the fourth-order elliptic perturbation problem and the Helmholtz transmission eigenvalue problem, and present their respective convergence estimation. The analysis can be somehow alike with existing works, provided the optimal approximation has been obtained in the previous section, and we will not have the readers involved too much in the details. Though, for the fourth-order elliptic perturbation problem, the inhomogeneous coefficient seems not studied yet; we thus include the technical proofs of Theorem \ref{thm:errorRRM}, Lemma \ref{lem:regularity} and Theorem \ref{thm:uniform convergence} in Section \ref{sec:pf4ep} for the ease of the readers who is interested in the effect of the coefficients.

\subsection{Robust RRM scheme for fourth order elliptic perturbation problem}
\label{sec:rrmep}
The finite element scheme for \eqref{eq: pertubation_variational_form} is to find $u_h^{\rm{R}} \in V_{h0}^{\rm{R}}$ such that  
\begin{align}\label{eq:discrete form RRM}
\varepsilon^{2}a_{h}(u_{h}^{\rm{R}},v_{h}) + b_{h}(u_{h}^{\rm{R}},v_{h}) = (f,v_{h}), \quad \forall v_{h} \in V_{h0}^{\rm{R}}.
\end{align}
where $a_h(u_h^{\rm{R}}, v_h) := (\beta(\boldsymbol{x})\Delta_h u_h^{\rm{R}}, \Delta_h v_h)$ and $b_h(u_h^{\rm{R}}, v_h) = (\nabla_h u_h^{\rm{R}}, \nabla_h v_h)$. 

We define a discrete energy norm on $V_{h}$ as
$$
\|w\|_{\varepsilon, h}:=\sqrt{\varepsilon^2 |w|^2_{2,h}+|w|^2_{1,h}}.
$$
With the boundedness of $\beta	(x)$ and \eqref{eq: discrete strengthened Miranda-Talenti estimate},  it's easy to check the well-posedness of the discrete weak formulation by Lax-Milgram theorem. 


%
\begin{theorem}\label{thm:errorRRM}
Let $u$ and $u_{h}^{\rm R}$ be the solutions of \eqref{eq:model problem} and \eqref{eq:discrete form RRM}, respectively. For $ u~\in~H^{2}_{0}(\Omega) \cap~H^{3}(\Omega)$, it holds that
\begin{equation}
\| u-u_{h}^{\rm{R}} \|_{\varepsilon,h} \lesssim h |u|_{2,\Omega}  + \varepsilon h (|u|_{2,\Omega} + |u|_{3,\Omega}) + \varepsilon^{2} h   (|u|_{2,\Omega} + |u|_{3,\Omega}+ \|\Delta^{2}u\|_{0,\Omega}).
\end{equation}
Moreover, if further $\mathcal{G}_{h}$ is uniform, then 
\begin{equation}\| u-u_{h}^{\rm{R}} \|_{\varepsilon,h} \lesssim h^{2}|u|_{3,\Omega} + \varepsilon h (|u|_{2,\Omega} + |u|_{3,\Omega}) + \varepsilon^{2} h   (|u|_{2,\Omega} + |u|_{3,\Omega}+ \|\Delta^{2}u\|_{0,\Omega}).
\end{equation}
\end{theorem}

From Theorem~\ref{thm:errorRRM}, the RRM element, which uses piecewise quadratic functions, ensures linear convergence in the energy norm as long as $|u|_{2,\Omega}$ , $|u|_{3,\Omega}$ and $\|\Delta^2 u\|_{0,\Omega}$ are uniformly bounded. When $\varepsilon$ approaches zero, the convergence rate in the energy norm approaches $O(h^{2})$ on uniform grids if $u$ is sufficiently smooth.  

However, $|u|_{2,\Omega}$ and $|u|_{3,\Omega}$ may blow up when $\varepsilon$ tends to zero. When $\beta$ is a constant, the result below is the one given in \cite{Nilssen;Tai;Winther2001}. Let $u^{0}$ be the solution of the following boundary value problem:
\begin{equation}\label{eq:2nd problem}
\left\{
\begin{array}{rl}
  -\Delta u^{0} = f, & \mbox{ in } \ \Omega , \\
u^{0}   = 0, & \mbox{ on } \ \partial\Omega .
\end{array}
\right.
\end{equation}
Next we will show the uniform result by the  regularity. 
\begin{lemma}\label{lem:regularity}
For a convex domain $\Omega$, there exist a constant $C$, independent of $\varepsilon$ and $f$, such that
\begin{align}
|u|_{2,\Omega} + \varepsilon |u|_{3,\Omega} & \leqslant C \varepsilon^{-\frac{1}{2}} \|f\|_{0,\Omega}; \\
 \big|u- u^{0}\big|_{1,\Omega} & \leqslant C \varepsilon^{\frac{1}{2}} \|f\|_{0,\Omega}.
\end{align}
\end{lemma}

\begin{theorem}\label{thm:uniform convergence}
Let $\Omega$ be convex and $f\in L^{2}(\Omega)$. Then 
\begin{equation}
\| u-u_{h}^{\rm{R}}\|_{\varepsilon,h}\lesssim h^{\frac{1}{2}}\|f\|_{0,\Omega} + \varepsilon h^{\frac{1}{2}}\|f\|_{0,\Omega}.
\end{equation}
\end{theorem}

\subsection{Optimal scheme for Helmholtz transmission eigenvalue problem}
In this part, we apply the RRM scheme for Helmholtz transmission eigenvalue problem. The discrete weak formulation corresponding to \eqref{eq: variational_form_Helm} is to  find $(\tau_h, u_h)\in \mathbb{R}\times V_{h0}^{\rm R}$ such that $\mathcal{B}_h(u_h, u_h) = 1$ and
\begin{equation} \label{eq: discrete_transmission_eigs}
	\mathcal{A}_{\tau_h,h}(u_h, v_h)  = \tau_h \mathcal{B}_h(u_h, v_h).
\end{equation}


\begin{theorem}  \label{thm: convergence rate of Helm_transmission}
Let $(\tau,u)$, $(\tau_h,u_h)$ be the solution of \eqref{eq: variational_form_Helm} and \eqref{eq: discrete_transmission_eigs}. Let $u\in H^3(\Omega) \cap H_0^2(\Omega)$ and the domain be convex. Under the assumptions of Lemma 3.2 in \cite{sun2011iterative}, then we can obtain the following approximate results
\begin{equation}
|u - u_h|_{2,h} \lesssim h,\quad	|\tau - \tau_h|\lesssim h^2.
\end{equation}
\end{theorem}
	
\begin{proof}
Following \cite[Theorem 12] {xi2020high}, let $u$ be the solution of bi-Laplace equation \eqref{eq:inbl}. 
then the corresponding discrete formulation is to find $u_h^{\rm{R}} \in V_{h0}^{\rm{R}}$ such that
$$
(\beta(\boldsymbol{x})\Delta_h u_h^{\rm{R}}, \Delta_h v_h) = (f,v_h), \forall v_h \in V_{h0}^{\rm{R}}.
$$
It's straightforward to derive that 
\begin{equation}  \label{eq:error_biLap}
|u - u_h^R|_{2,h} \lesssim h|u|_{3,\Omega}.
\end{equation}
Combining \eqref{eq:error_biLap} and the classical theory of nonconforming finite element method (c.f. \cite{babuska1991eigenvalue}), this theorem can be established.
\end{proof}

\paragraph{\bf Numerical implementation} 
We follow the techniques from \cite{xi2020high}. Let $\{\varphi\}_{j = 1}^{N}$ be a basis for $v_{h0}^{\rm R}$ and the corresponding FEM solution $u_h = \sum_{j = 1}^N  u_j\varphi_j$. We need the following matrices in the discrete case and obtain the discretized quadratic eigenvalue problem
\begin{equation} \label{eq: eigs_quadratic}
	(A + \tau_h B + \tau_h^2 C)x = 0,
\end{equation}
where $x = (u_1,...,u_N)^T$. The computation of matrices $A, B, C$ involves numerical integration of basis functions with non-constant coefficients.

\begin{table}[!ht]
\centering
	\begin{tabularx}{12cm}{p{1cm}<{\centering} p{4cm}<{\centering} X<{\centering}     }
		 \text{Matrix} & \text{Dimension} & Definition \\
		\midrule
		A & $N\times N$ & $A_{i,j} = \int_\Omega \frac{1}{\beta-1}\Delta\varphi_i\Delta\varphi_j $ \\
		B & $N\times N$ & $B_{i,j} = \int_\Omega \frac{1}{\beta-1}\Delta\varphi_i\varphi_j+\frac{1}{\beta-1}\varphi_i\Delta\varphi_j - \nabla \varphi_i \cdot \nabla \varphi_j$ \\
		C & $N\times N$ & $C_{i,j} = \int_{\Omega} \frac{\beta}{\beta - 1}\varphi_i\varphi_j$\\
	\end{tabularx}

\end{table}
For \eqref{eq: eigs_quadratic}, in practical computation, we convert it to the linear eigenvalue problem
$$
\left(\begin{array}{cc}
	-B & -A\\
	I & 0 
\end{array}\right)
\left(  \begin{array}{c}
	p_1\\
	p_2
\end{array}      \right) = \tau_h \left(\begin{array}{cc}
	C & 0\\
	0 & I 
\end{array}\right)
\left(  \begin{array}{c}
	p_1\\
	p_2
\end{array}      \right) ,
$$
where $p_2=x, p_1=\tau x$ and use MATLAB function "eigs" to solve. 

\section{Numerical experiments}
\label{sec:experiments}

We consider uniform subdivisions as well as non-uniform subdivisions. The series of nonuniform subdivisions are obtained by firstly subdividing $\Omega$ to a series of finer and finer uniform subdivisions, and then refining once each uniform subdivision by a same ratio to obtain a non-uniform subdivision.  Figure~\ref{fig:non-uniformgrid}  illustrates how non-uniform subdivisions are generated.  Numerical examples of the model problem~\eqref{eq:model problem} are given below. 
\begin{figure}[!htbp]
\centering
\includegraphics[height=0.29\hsize]{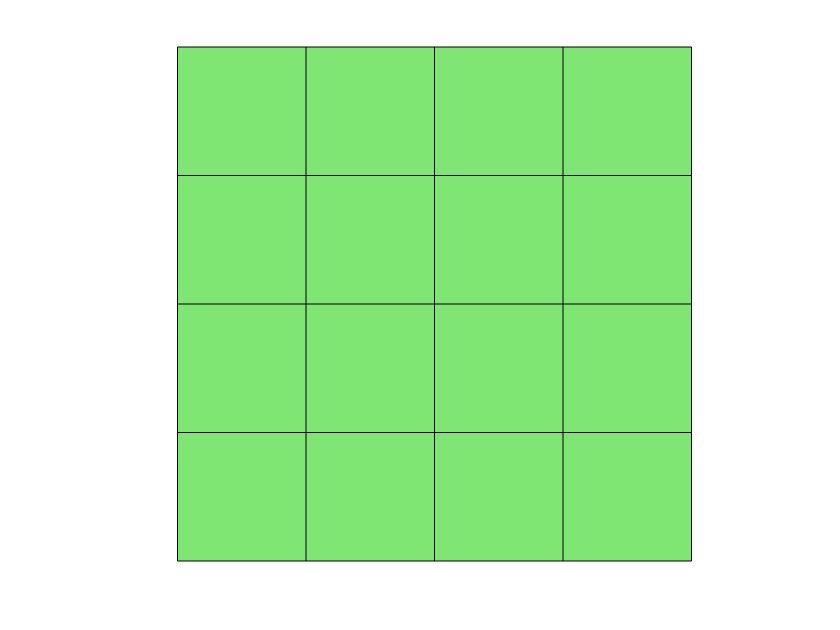}
\qquad
\includegraphics[height=0.29\hsize]{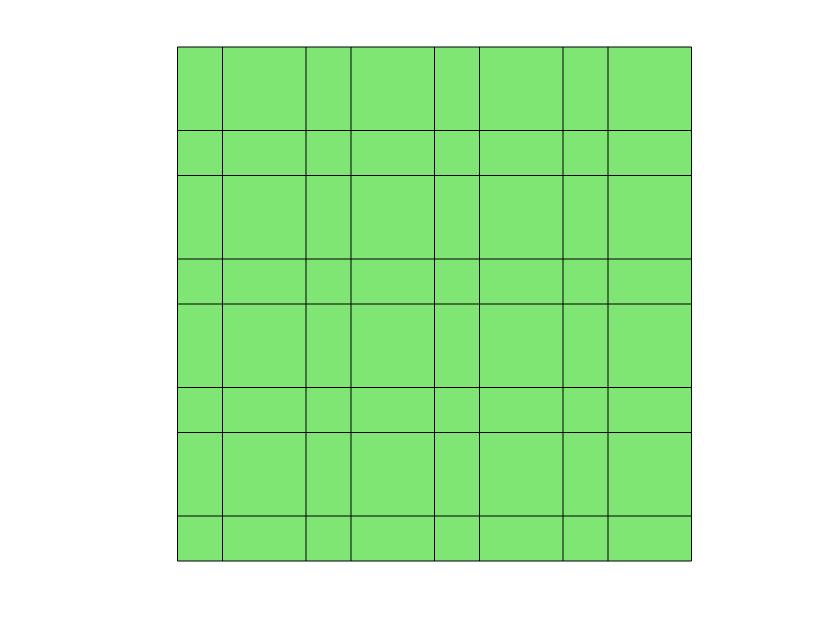}
\caption{To generate a sequence of non-uniform subdivisions, we generate firstly a sequence of uniform subdivisions, and refine each uniform subdivision by a same ratio to generate a non-uniform subdivision. See the illustration on square domains. }\label{fig:non-uniformgrid}
\end{figure}

\subsection{ Numerical examples for fourth order elliptic perturbation problem      }

\begin{ex}  \label{ex: perturb_ex_5}
Let $\Omega=(0,1)^{2}$. Take $u = \big(sin(\pi x)sin(\pi y)\big)^{2}, \beta(\boldsymbol{x}) = 8 + x - y$ and $f = \varepsilon^{2} \Delta(\beta(\boldsymbol{x}) \Delta u) - \Delta u$. Then $u$ is the solution of problem~\eqref{eq:model problem}. Apply \eqref{eq:discrete form RRM}  to get the discrete solution $u_{h}^{\rm{R}}$ on uniform or non-uniform meshes, and compute the relative energy error $\frac{\| u-u_{h}^{\rm{R}} \|_{\varepsilon,h}}{\| u \|_{\varepsilon,h}}$. From the experiment result of the non-uniform case in the left part of Figure~\ref{fig:table1}, the convergence rate is $\mathcal{O}(h)$ for $0<\varepsilon \leq1$. From the experiment result of the uniform case in the right part of  Figure~\ref{fig:table1}, the convergence rate is $\mathcal{O}(h)$ when $\varepsilon = \mathcal{O}(1)$ and $\mathcal{O}(h^{2})$ when $\varepsilon \ll 1 $. Both cases verify the theoretical findings in Theorem \ref{thm:errorRRM}. Figure \ref{fig:error_surface} shows the numerical solution and relative error in the surface.

\begin{figure}[!htbp]  
\centering
\includegraphics[scale=0.21]{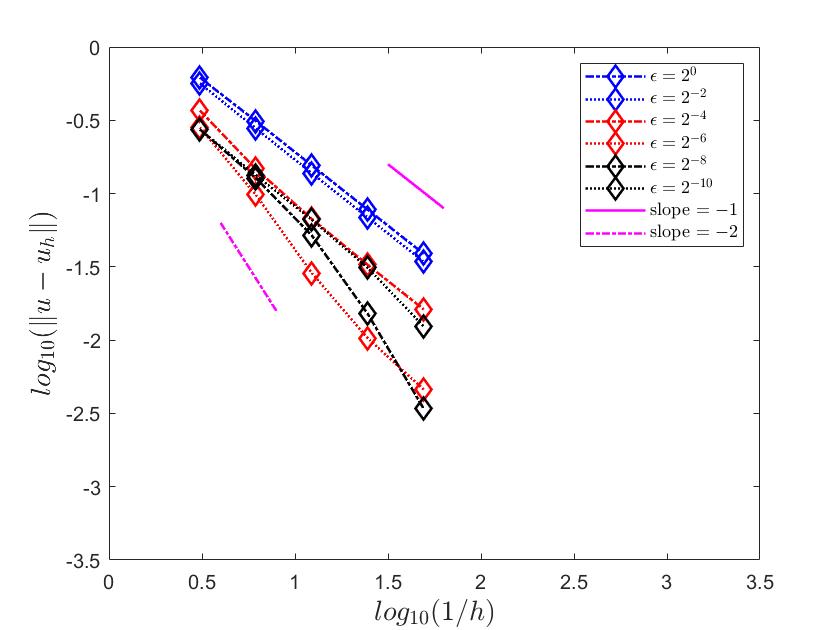}
\qquad
\includegraphics[scale=0.21]{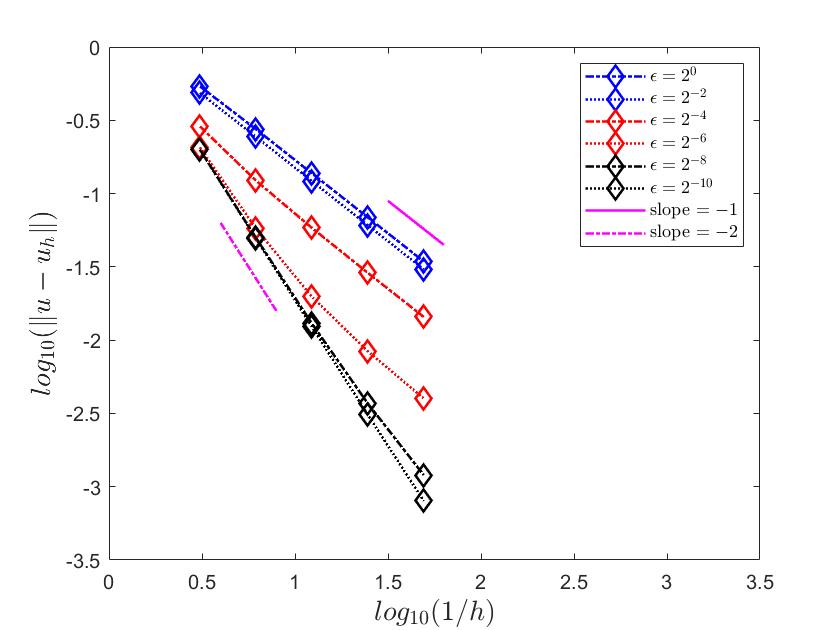}
\caption{Left:Relative error $\frac{\| u-u_{h}^{\rm{R}} \|_{\varepsilon,h}}{\| u \|_{\varepsilon,h}}$ on non-uniform grids in Example \ref{ex: perturb_ex_5}. Right:Relative error $\frac{\| u-u_{h}^{\rm{R}} \|_{\varepsilon,h}}{\| u \|_{\varepsilon,h}}$ on uniform grids in Example \ref{ex: perturb_ex_5}.} \label{fig:table1}
\end{figure}

\begin{figure}[!htbp] 
\centering
\includegraphics[scale=0.15]{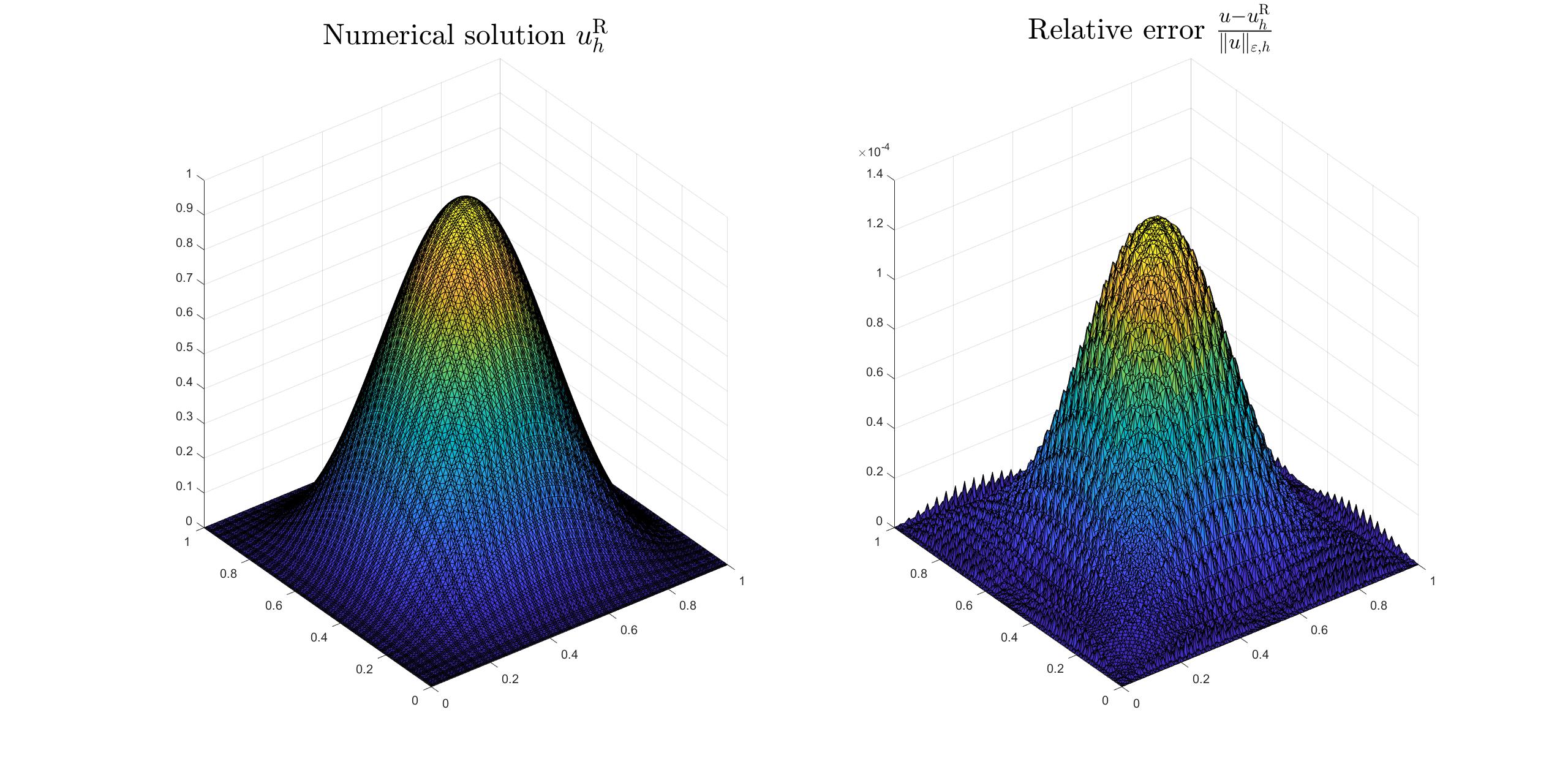}
\caption{Plot of the numerical solution on uniform mesh when  $h = 2^{-6}$ and $\epsilon = 2^{-2}$ }. \label{fig:error_surface}
\end{figure}

\end{ex} 

\begin{ex} \label{ex: perturb_ex_6}
Let $\Omega = (0,2)^{2}\backslash [1,2]^{2}$. Take the same $u, \beta$ and $f$ as in  Example \ref{ex: perturb_ex_5}. From Figure~\ref{fig:table3}, the convergence rates on the L-shaped domain are consistent with the results derived on $\Omega = (0,1)^{2}$.  It verifies the theoretical findings, especially the results in Theorem~\ref{thm:approxH02}, which shows that the interpolating properties are valid for non-convex domains. Figure \ref{fig:error_surface_2} shows the numerical solution and relative error in the  surface.  
\end{ex}
 
\begin{figure}[!htbp]   
\centering
\includegraphics[scale=0.2]{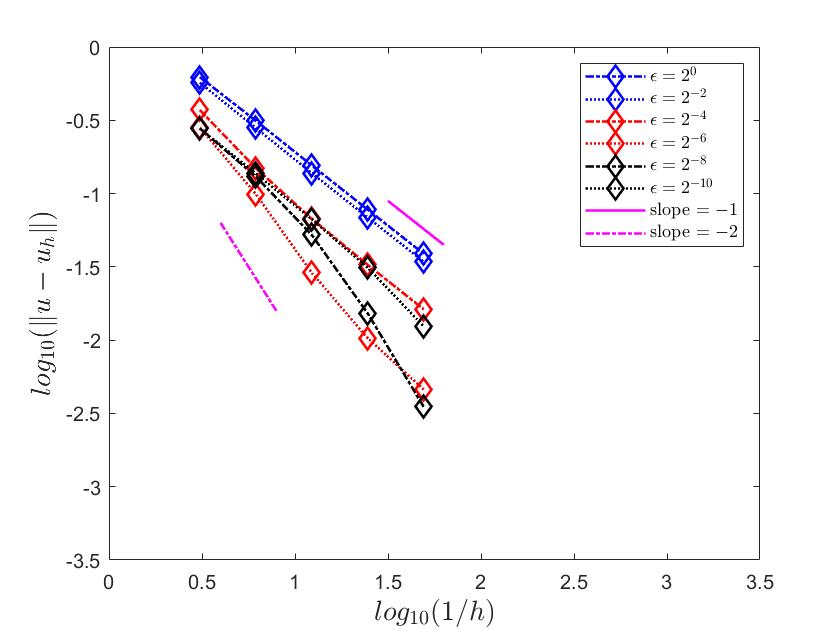}
\qquad
\includegraphics[scale=0.2]{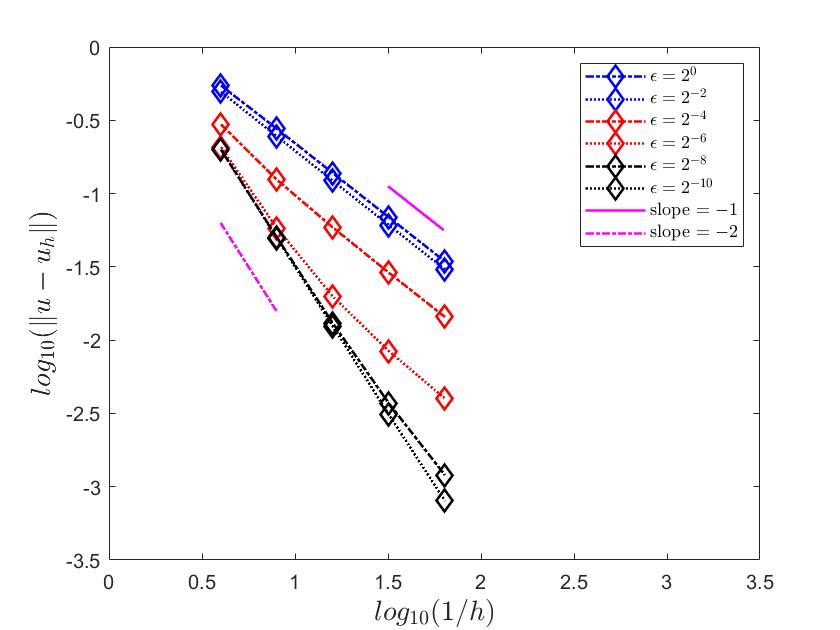}
\caption{Left:Relative error $\frac{\| u-u_{h}^{\rm{R}} \|_{\varepsilon,h}}{\| u \|_{\varepsilon,h}}$ on non-uniform grids in Example \ref{ex: perturb_ex_6}. Right:Relative error $\frac{\| u-u_{h}^{\rm{R}} \|_{\varepsilon,h}}{\| u \|_{\varepsilon,h}}$ on uniform grids in Example \ref{ex: perturb_ex_6}.} \label{fig:table3}
\end{figure}

\begin{figure}[!htbp] 
\centering
\includegraphics[scale=0.15]{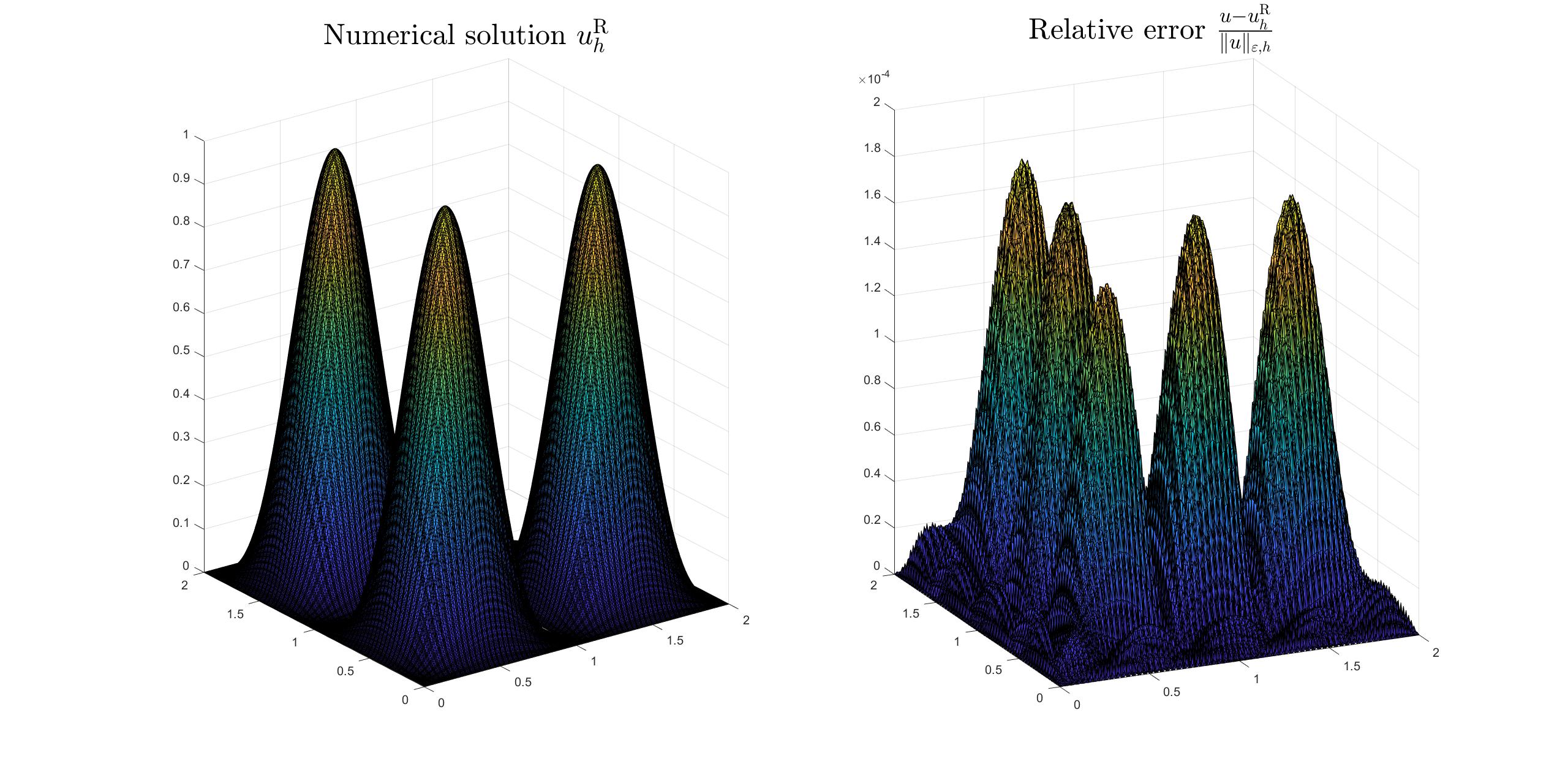}
\caption{Plot of the numerical solution on uniform mesh when  $h = 2^{-6}$ and $\epsilon = 2^{-2}$ }. \label{fig:error_surface_2}
\end{figure}

\begin{ex} \label{ex: perturb_ex_11}
	Let $\Omega=(0,1)^{2}$. Consider \eqref{eq:model problem} with $f = 2\pi^2sin(\pi x)sin(\pi y)$ and $\beta = 8 + x - y$.  The explict expression of $u$ is unknown, but the exact solution of \eqref{eq:2nd problem} reads $u^{0} = sin(\pi x)sin(\pi y)$. Here we take $\varepsilon$ to be small enough.  From Theorem~\ref{thm:uniform convergence}, $\| u^{0}-u_{h}^{\rm{R}} \|_{\varepsilon,h} \leq \| u^{0}-u \|_{\varepsilon,h} + \| u-u_{h}^{\rm{R}} \|_{\varepsilon,h}$, so the convergence rate of the   error  is $\mathcal{O}(h^{\frac{1}{2}})$ when the mesh is relatively coarse, which is shown in Figure~\ref{fig:table5}. Figure \ref{fig:error_surface_3} shows the numerical solution and relative error in the surface.  

\end{ex}

\begin{figure}[!htbp]    
\centering
\includegraphics[scale=0.2]{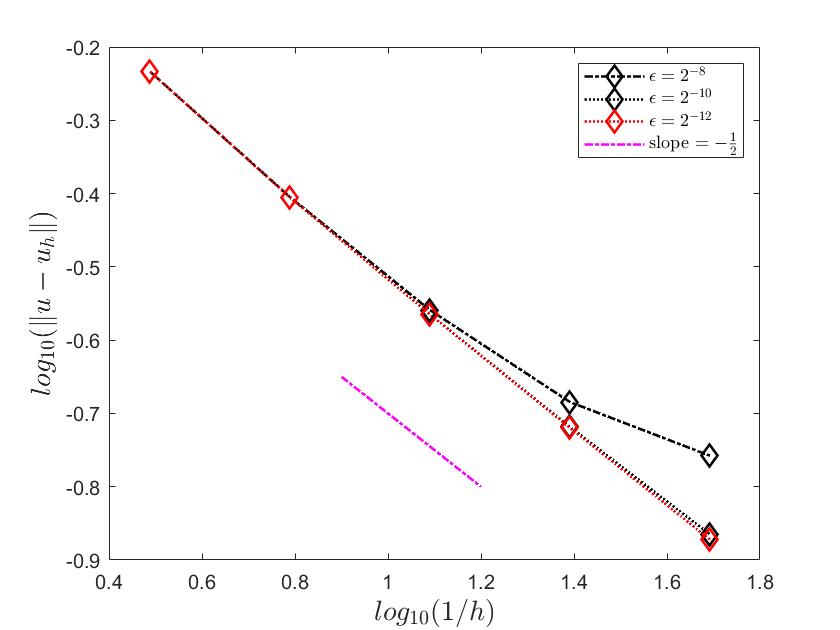}
\qquad
\includegraphics[scale=0.2]{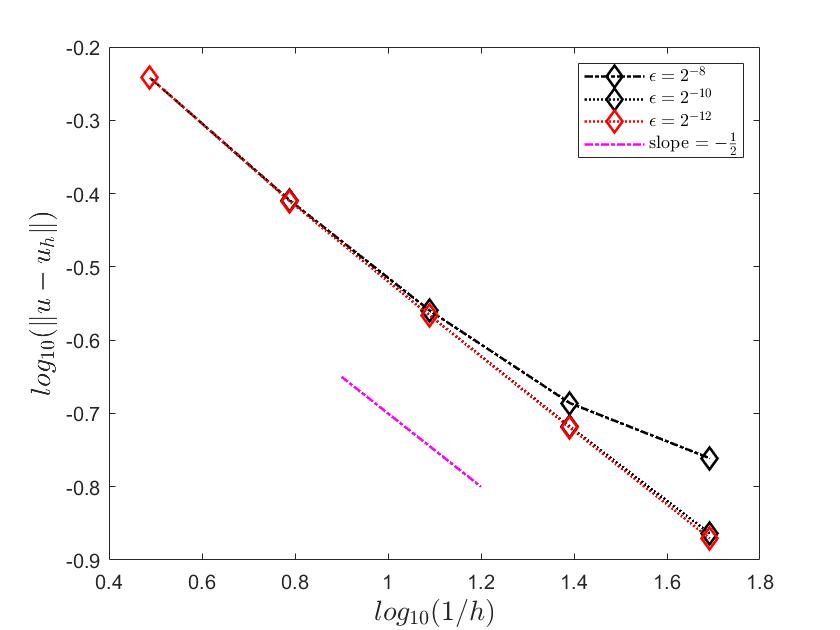}
\caption{Left:Relative error $\frac{\| u^0-u_{h}^{\rm{R}} \|_{\varepsilon,h}}{\| u^0 \|_{\varepsilon,h}}$ on non-uniform grids in Example \ref{ex: perturb_ex_11}. Right:Relative error $\frac{\| u^0-u_{h}^{\rm{R}} \|_{\varepsilon,h}}{\| u^0 \|_{\varepsilon,h}}$ on uniform grids in Example \ref{ex: perturb_ex_11}. Both figures illustrate that $\| u^0-u_{h}^{\rm{R}} \|_{\varepsilon,h}$ decays in order $\mathcal{O}(h^{-1/2})$ when the mesh is relatively coarse.} \label{fig:table5}
\end{figure}

\begin{figure}[!htbp]  
\centering
\includegraphics[scale=0.15]{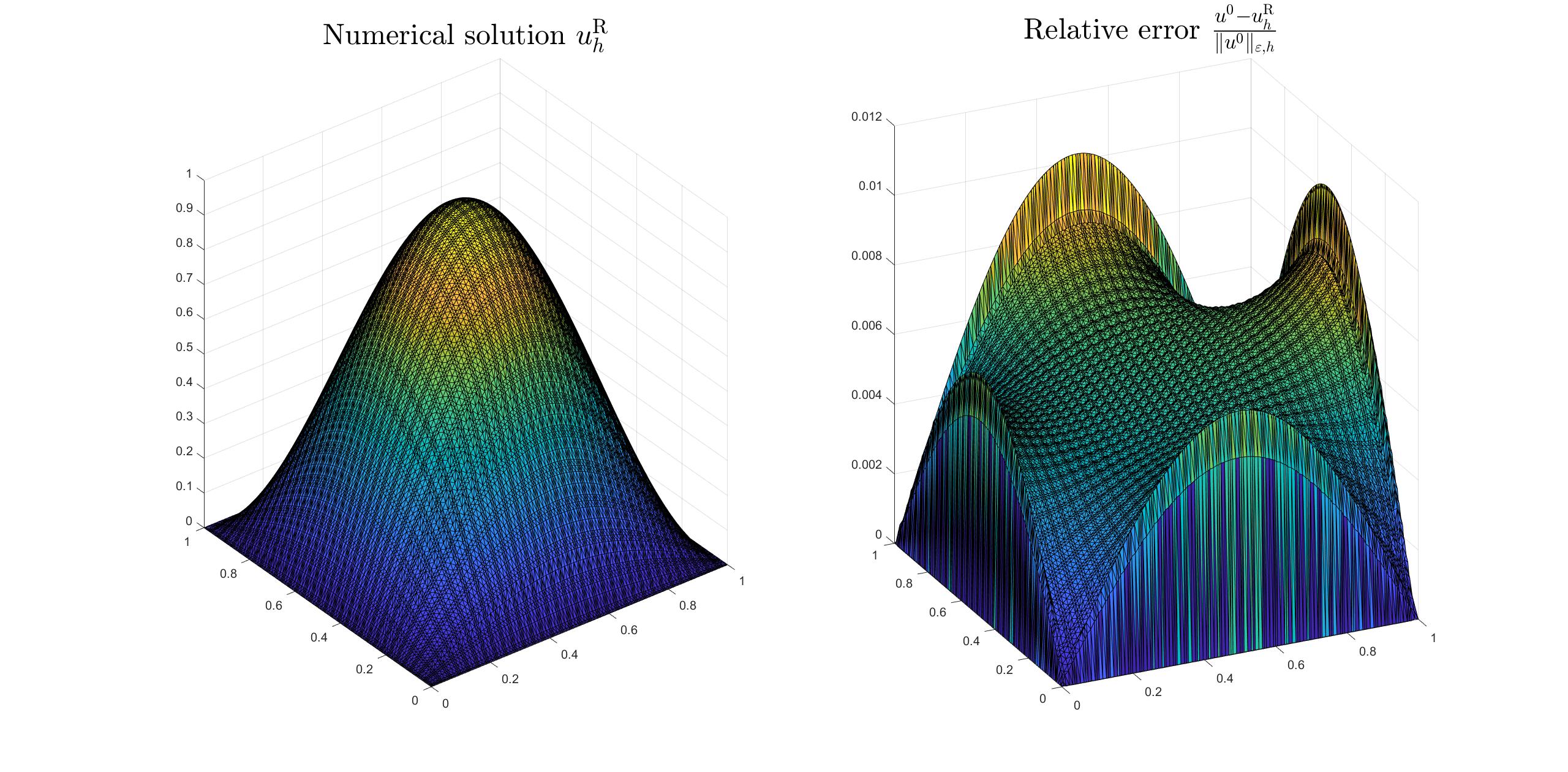}
\caption{Plot of the numerical solution on uniform mesh when  $h = 2^{-6}$ and $\epsilon = 2^{-10}$ }. \label{fig:error_surface_3}
\end{figure}

\subsection{Numerical examples for Helmholtz transmission eigenvalue problem}
Here we focus on the case $\beta(\boldsymbol{x})> 1$ which is of dominant interest in practice \cite{colton1998inverse}. For $0<\beta(\boldsymbol{x})<1$, it can be treated similarly. We refine the mesh uniformly for all examples. For each series of meshes, we show the lowest six eigenvalues $\{\lambda_{h_k}\}_{k = 1}^6$($\lambda_h = \sqrt{\tau_h}$). The convergent orders are computed by 
$$
log_2(|\frac{\lambda_{h_l} - \lambda_{h_{l+1}}}{\lambda_{h_{l+1}} - \lambda_{h_{l+2}} }|).
$$

\begin{ex} \label{ex: Helmholtz_trans_2}
Consider the model problem \eqref{eq: Helmholtz_transmission_model} on square domain $\Omega = (0,1)^2$ with $\beta(\boldsymbol{x}) = 8 + x - y$. The numerical results are showed in Table \ref{table: Helmholtz_trans_2}. On each mesh level we show the first six real eigenvalues and the convergence rate  is $\mathcal{O}(h^2)$. The result verifies the theoretical findings in  theorem  \ref{thm: convergence rate of Helm_transmission}.
\end{ex}

\begin{table}[!htbp]
	\renewcommand{\tablename}{Table}
	\centering
	\caption{The performance of $V_{h0}^{\rm R}$ in Example \ref{ex: Helmholtz_trans_2}}
	\label{table: Helmholtz_trans_2}
     \begin{tabular}{llllllll}
		 $\lambda \backslash h$ & $2^{-5}$ & $2^{-6}$ & $2^{-7}$ & $2^{-8}$& $2^{-9}$  & \text{Trend}& \text { Rate }\\
		\midrule
		$\lambda_1$ & 2.825272 & 2.822959 & 2.822382 & 2.822375 & 2.822201 &$\searrow$ & 2.00 \\
		$\lambda_2$ & 3.547819 & 3.540972 & 3.539265 & 3.538839 & 3.538732 &$\searrow$ & 2.00 \\
		$\lambda_3$ & 3.548079 & 3.541258 & 3.539558 & 3.539133 & 3.539027 &$\searrow$ & 2.00 \\
		$\lambda_4$ & 4.122230 & 4.118872 & 4.118025 & 4.117813 & 4.117760 &$\searrow$ & 1.99 \\
		$\lambda_5$ & 4.527683 & 4.508192 & 4.503343 & 4.502132 & 4.501830 &$\searrow$ & 2.00\\
		$\lambda_6$ & 5.003727 & 4.992792 & 4.990053 & 4.989368 & 4.989196 &$\searrow$ & 2.00 \\
	\end{tabular}
\end{table}

\begin{ex} \label{ex: Helmholtz_trans_4}
Consider the model problem \eqref{eq: Helmholtz_transmission_model} on square domain $\Omega = (0,2)^2\backslash(1,2)^2$ with $\beta(\boldsymbol{x}) = 8 + x - y$. The numerical results are showed in Table \ref{table: Helmholtz_trans_4}. On each mesh level we show the lowest six real eigenvalues and the convergence rate  is $\mathcal{O}(h^2)$. The result verifies the theoretical findings in  theorem  \ref{thm: convergence rate of Helm_transmission}.
\end{ex}

\begin{table}[!htbp]
	\centering 
	\renewcommand{\tablename}{Table}
	\caption{The performance of $V_{h0}^{\rm R}$ in Example \ref{ex: Helmholtz_trans_4}}
	\label{table: Helmholtz_trans_4}
	 \begin{tabular}{llllllll}
		 $\lambda \backslash h$ & $2^{-4}$ & $2^{-5}$ & $2^{-6}$ & $2^{-7}$& $2^{-8}$  & \text{Trend}& \text { Rate }\\
		\midrule
		$\lambda_1$ & 2.186708 & 2.182840 & 2.181880 & 2.181643 & 2.181584 &$\searrow$ & 2.01 \\
		$\lambda_2$ & 2.305447 & 2.294871 & 2.292233 & 2.291574 & 2.291409 &$\searrow$ & 2.00 \\
		$\lambda_3$ & 2.570287 & 2.562290 & 2.560272 & 2.255977 & 2.559640 &$\searrow$ & 1.99 \\
		$\lambda_4$ & 2.719306 & 2.711826 & 2.709925 & 2.709450 & 2.709331 &$\searrow$ & 1.99 \\
		$\lambda_5$ & 2.993792 & 2.977453 & 2.973325 & 2.977229 & 2.972031 &$\searrow$ & 1.99\\
		$\lambda_6$ & 3.195727 & 3.166438 & 3.158732 & 3.156780 & 3.156289 &$\searrow$ & 1.97 \\
	\end{tabular}
\end{table}

\appendix 

\section{Some technical proofs}

\subsection{Proofs of Proposition \ref{prop:scaling} and Proposition \ref{pro:property of phi 5x5} in Section \ref{sec:rrmrv}}
\label{sec:proofs3.1}

\subsubsection{\bf Proof of Proposition \ref{prop:scaling}}
\label{sec:pfprop:scaling}
For $T\subset \mathcal{M}_{K}$, $\varphi_{K}|_{T}\in P_{2}(T)$ and it can be written as
\begin{equation}\label{eq:RRMbasisOnK}
\varphi_{K}|_{T} = \sum_{s=1:4}b_{s}p_{s}^{\rm{M}} + \sum_{t=5:8}b_{t}q_{t}^{\rm{M}},
\end{equation}
where $p_{s}^{\rm{M}}$ and $q_{t}^{\rm{M}}$ represent the rectangular Morley basis functions related to nodes and edges of $T$, respectively. 
It is known that
\begin{equation}\label{eq:normRMbasis}
\big|p_{s}^{\rm{M}}\big|_{k,T} \leq C_{\gamma_{0}}h_{T}^{1-k} \quad \mbox{and} \quad  \big|q_{t}^{\rm{M}}\big|_{k,T} \leq C_{\gamma_{0}}h_{T} h_{T}^{1-k}.
\end{equation}
From \eqref{eq:values1Basis}--\eqref{eq:values3Basis}, there exists a constant $C_{\gamma_{0}}$, such that for $s, \, t = 1:4$,
\begin{equation}\label{eq:valuesOnK}
b_{s} \leq C_{\gamma_{0}} \ \mbox{ and }\ b_{t} \leq C_{\gamma_{0}}h_{T}^{-1}.
\end{equation}
A combination of \eqref{eq:RRMbasisOnK}, \eqref{eq:normRMbasis}, and \eqref{eq:valuesOnK} leads to the desired result.\qed

\subsubsection{\bf Proof of Proposition \ref{pro:property of phi 5x5}}
\label{sec:pfpro:property of phi 5x5}
$(a)$
Consider the first equality in $(a)$. The function on its left-hand-side is a sum of $P_{2}$ polynomials restricted on $K$, and the function on the right-hand-side is a bilinear polynomial. Utilizing \eqref{eq:values1Basis}, \eqref{eq:values2Basis}, \eqref{eq:values3Basis}, and $\varphi_{T}(X_{1,1}^{T}) = \frac{L_{T,-1}}{L_{T,-1}+L_{T}}\cdot \frac{H_{T,-1}}{H_{T,-1}+H_{T}}$, for any $T\in A_{K}$, it is calculated directly that, for the left-hand-side and the right-hand-side functions, their values on the vertices of $K$ and normal derivatives on the midpoints of edges on $\partial K$ are equal. Therefore, the first equality in $(a)$ is valid. 

The second equality is obtained by $v(x_{T},y_{T}) = \fint_{T}v \ud x d y$ for any $v\in Q_{1}(\mathcal{M}_{K})$.
 
\noindent$(b)$
For $v = x^{2}$, direct calculation leads to, \begin{align*}
& \sum_{T\in A_{K}}  v(c_{T})\,\varphi_{T}(x,y) = x^{2} + \frac{1}{4}\sum_{T\in A_{K}} L_{T}^{2}\,\varphi_{T}(x,y), \quad \forall (x,y)\in K; \\
& \sum_{T\in A_{K}} (\fint_{T}v{\,d x d y }) \,\varphi_{T}(x,y)  = x^{2} + \frac{1}{3}\sum_{T\in A_{K}} L_{T}^{2}\,\varphi_{T}(x,y), \quad \forall (x,y)\in K.
\end{align*}
For $v = y^{2}$, we have similarly,
\begin{align*}
& \sum_{T\in A_{K}} v(c_{T})\,\varphi_{T}(x,y) = y^{2} + \frac{1}{4}\sum_{T\in A_{K}} H_{T}^{2}\,\varphi_{T}(x,y), \quad \forall (x,y)\in K; \\
& \sum_{T\in A_{K}} (\fint_{T}v{\,d x d y })\, \varphi_{T}(x,y)  = y^{2} + \frac{1}{3}\sum_{T\in A_{K}} H_{T}^{2}\,\varphi_{T}(x,y), \quad \forall (x,y)\in K.
\end{align*}
Therefore, for any function $v \in P_{2}(\mathcal{M}_{K})$ and $ (x,y)\in K$, it holds that, 
\begin{align*}
&  \sum_{T\in A_{K}} r_{T}(v)\,\varphi_{T}(x,y) = v(x,y), \ \mbox{ where }\ r_{T}(v) = v(c_{T}) - \tfrac{1}{8}\big(L_{T}^{2}\,\tfrac{\partial^{2}v}{\partial x^{2}} + H_{T}^{2}\,\tfrac{\partial^{2}v}{\partial y^{2}} \big);
\\
& \sum_{T\in A_{K}} t_{T}(v)\,\varphi_{T}(x,y) = v(x,y), \ \mbox{ where }\  t_{T}(v) = \fint_{T}v{\,d x d y} - \tfrac{1}{6}\big(L_{T}^{2}\,\tfrac{\partial^{2}v}{\partial x^{2}}  + H_{T}^{2}\,\tfrac{\partial^{2}v}{\partial y^{2}} \big).
\end{align*}

\noindent$(c)$
Consider the function on the left-hand-side of this equality. Direct calculation leads to that, the function values on the vertices of $K$ and normal derivatives on the midpoints of edges on $\partial K$ equals to zero. Therefore, $\sum_{T\in A_{K}} (d_{T}L_{T}H_{T})\,\varphi_{T}$ is a zero function restricted on $K$. \qed


%
%
\subsection{Proof of Lemma \ref{lemma:ccc} in Section \ref{sec:aied}}
\label{sec:pflm5}
$(a)$
By Lemma~\ref{lem:property of phi} and the difference theory, we replace the second derivatives appearing in the expression of $t_{K}(v)$ with a weighted sum of five integral mean values around $K$, where the weights are computed to be $\big\{w_{\mu}^{K}\big\}_{\mu = 1:5}$. That is to say,  for $v\in L^1(\widetilde{\Omega}_h)$ such that $v|_{\Delta_T}\in P_{2}(\Delta_T)$, it holds that $\lambda_{K}(v) = t_{K}(v)$. Therefore, $(\widetilde{\Pi}_{h} v)|_{T} = v|_{T}$ for any $T \in \mathcal{G}_{h}$.

\noindent$(b)$
From the assumption of local quasi-uniformity in \eqref{eq:regularity}, we can conclude that all cells in $\Delta_{T}$ are of comparable size. Utilizing ~\eqref{eq:normEstimate}, we have $\max_{ \substack{ \mathcal{M}_{K}\supset T \\ K\in \mathcal{J}_{h}} } |\varphi_{K} |_{k,T} \lesssim h_{T}^{1-k}$, where the hidden constant depends only  on~$\gamma_{0}$.

\begin{align*}
|\widetilde{\Pi}_{h}v|_{k,T}^{2} & = \big|\sum_{ \substack{ \mathcal{M}_{K}\supset T \\ K\in \mathcal{J}_{h}} } \lambda_{K}(v)\,\varphi_{K}(x,y)\big|_{k,T}^{2} 
 \lesssim \sum_{ \substack{ \mathcal{M}_{K}\supset T \\ K\in \mathcal{J}_{h}} } |\lambda_{K}(v)|^{2} | \,\varphi_{K}(x,y)|_{k,T}^{2} \\ 
& \leq \max_{ \substack{ \mathcal{M}_{K}\supset T \\ K\in \mathcal{J}_{h}} } |\varphi_{K}(x,y)|_{k,T}^{2} \sum_{ \substack{ \mathcal{M}_{K}\supset T \\ K\in \mathcal{J}_{h}} } |\lambda_{K}(v)|^{2} 
\lesssim h_{T}^{2-2k} \sum_{ \substack{ \mathcal{M}_{K}\supset T \\ K\in \mathcal{J}_{h}} }\Big( \sum\limits_{\mu = 1}^{5}w_{\mu}^{K}\fint_{S_{\mu}^{K}}v{\,d x d y} \Big)^{2} \\
& \lesssim h_{T}^{2-2k} \sum_{ \substack{ \mathcal{M}_{K}\supset T \\ K\in \mathcal{J}_{h}} } \sum\limits_{\mu = 1}^{5}\big(\fint_{S_{\mu}^{K}}v{\,d x d y} \big)^{2}
 \lesssim h_{T}^{2-2k} \sum_{ \substack{ \mathcal{M}_{K}\supset T \\ K \in \mathcal{J}_{h}} } \sum\limits_{\mu = 1}^{5}\frac{1}{|S_{\mu}^{K}|}\int_{S_{\mu}^{K}}v^{2} {\,d x d y}  \\
& \lesssim h_{T}^{-2k} \|v\|_{0,\Delta_{T}}^{2}.
\end{align*}
The proof is thus completed. 

\noindent$(c)$
For any polynomial $p \in P_{2}(\widetilde{\Omega}_{h})$, we have by Lemma \ref{lem:preserveP2}(a) and Lemma \ref{lem:stability}(b) that 
\begin{align*}
|v-\widetilde{\Pi}_{h}v|_{k,T} & \leq |v-p|_{k,T} + |\widetilde{\Pi}_{h}(p-v)|_{k,T} \\
& \lesssim |v-p|_{k,T} + h_{T}^{-k}\|v-p\|_{0,\Delta_{T}} 
\end{align*}

Since $\Delta_{T}$ is a finite union of rectangles, each of which is star-shaped ensured by \eqref{eq:regularity}, we can apply the Bramble-Hilbert lemma in the form presented in \cite{Dupont;Scott1980,Scott;Zhang1990} and obtain 
\begin{align}\label{eq:bestP2}
\inf_{p\in P_{2}}|v-p|_{k,\Delta_{T}} \lesssim h_{T}^{s-k}|v|_{s,\Delta_{T}}, \mbox{ with }  0\leq k \leq s \leq 3,
\end{align}
where the hidden constant is only dependent on $\gamma_{0}$.
Therefore, we derive
\begin{align*}
|v-\widetilde{\Pi}_{h}v|_{k,T}\lesssim h_{T}^{s-k}|v|_{s,\Delta_{T}}\mbox{ with } 0\leq k \leq s \leq 3.
\end{align*}
The proof is completed. \qed

\subsection{Proofs of Theorem \ref{thm:errorRRM}, Lemma \ref{lem:regularity} and Theorem \ref{thm:uniform convergence} in Section \ref{sec:rrmep}} 
\label{sec:pf4ep}

\subsubsection{\bf Proof of Theorem \ref{thm:errorRRM}} 
By the Strang's lemma~\cite{Ciarlet1978}, we have
\begin{align}\label{eq:StrangRRM}
\| u-u_{h}^{\rm R} \|_{\varepsilon,h} \lesssim \inf_{v_{h}\in V_{h0}^{\rm{R}}} \| u - v_{h} \|_{\varepsilon,h} + \sup_{w_{h}\in V_{h0}^{\rm{R}},w_{h} \ne 0}  \frac{E_{\varepsilon,h}(u,w_{h})}{\| w_{h}\|_{\varepsilon,h} },
\end{align}
where $E_{\varepsilon,h}(u,w_{h}) = \varepsilon^{2} a_{h}(u,w_{h}) + b_{h}(u,w_{h}) -(f,w_{h})$.

Consider the first term on the right-hand-side of~\eqref{eq:StrangRRM}, i.e., the approximation error. By Theorem~\ref{thm:approxH02}, the following two estimates are valid

\begin{equation}\label{lem:approx in energy norm}
\begin{split}
\inf_{v_{h}\in V_{h0}^{\rm{R}}} \| u - v_{h} \|_{\varepsilon,h} \leq \| u - \Pi_{h0}u \|_{\varepsilon,h} &\lesssim 
 |u - \Pi_{h0}u |_{1,h} + \varepsilon |u - \Pi_{h0}u |_{2,h}
\lesssim 
\left\{
\begin{array}{l}
h|u|_{2,\Omega} +\varepsilon h |u|_{3,\Omega}, \\
h^{2}|u|_{3,\Omega}+\varepsilon h |u|_{3,\Omega}.
\end{array}
\right.
\end{split}
\end{equation}

Consider the second term on the right-hand-side of~\eqref{eq:StrangRRM}, i.e., the consistency error. Let $\Pi_{h0}^{\rm{b}}$ be the nodal interpolation operator associated with the bilinear element, then
\begin{align*}
E_{\varepsilon,h}(u,w_{h}) & =  \varepsilon^{2}a_{h}(u,w_{h}) + b_{h}(u,w_{h}) - (f,w_{h}) \\
& = \varepsilon^{2} \sum_{K\in \mathcal{G}_{h}} \int_{K} \beta \Delta u \Delta w_{h}  + \sum_{K\in \mathcal{G}_{h}} \int_{K} \nabla u \cdot \nabla w_{h}  - \int_{\Omega}\left(\varepsilon^{2} \Delta(\beta \Delta u) - \Delta u\right)w_{h} \\
& = \varepsilon^{2} \sum_{K \in \mathcal{G}_h} \int_K \left( \beta(\boldsymbol{x})\Delta u\Delta w_h - \Delta(\beta(\boldsymbol{x})\Delta u)\Pi_{h0}^{\rm{b}} w_h \right) + \sum_{K\in \mathcal{G}_{h}} \int_{K} \left( \nabla u \cdot \nabla w_{h} + \Delta u w_{h}\right)\\
&  \quad + \varepsilon^{2}\int_{\Omega} \Delta(\beta(\boldsymbol{x})\Delta u)  \ (\Pi_{h0}^{\rm{b}} w_{h} - w_{h}) := R_{1} + R_{2} +R_{3}.
\end{align*}

Notice that $\Pi_{h0}^{\rm{b}} w_h\in H_0^1(\Omega)$, by Green's formula, 
\begin{align} \label{eq: bi_harm_proof_consist}
 R_{1} &= \sum_{K \in \mathcal{G}_h} \int_K \nabla(\beta(\boldsymbol{x})\Delta u)\cdot (\nabla w_h - \nabla \Pi_{h0}^{\rm{b}} w_h) -  \sum_{K \in \mathcal{G}_h} \int_{\partial K} \beta(\boldsymbol{x})\Delta u \frac{\partial w_h}{\partial \boldsymbol{n}}  \\
	& := T_1 + T_2 \nonumber 
\end{align}
By the Cauchy-Schwarz inequality and the approximation property of the interpolation operator $\Pi_{h0}^{\rm{b}}$,
\begin{equation} \label{eq: R_1_bi}
|T_1|\lesssim (|u|_{2,\Omega} + |u|_{3, \Omega})h|w_h|_{2,h}.
\end{equation}
\begin{equation}\label{eq: f_bi}
(f,  \Pi_{h0}^{\rm{b}} w_{h} - w_h) \lesssim h^2\|f\|_{0,\Omega}|w_h|_{2,h}
\end{equation}
Notice that the jumps of the mean values, over an interelement face, of the first order
derivatives of $w_h$ are all zero, and their mean values over a free face are zero, by a standard estimate in (e.g.,~\cite[Theorem 5.4.1]{Wang.M;Shi.Z2013mono} and \cite[(4.6)]{Nilssen;Tai;Winther2001}), we have
\begin{equation} \label{eq: R_2_bi}
|T_2| \lesssim h|u|_{3,\Omega}|w_h|_{2,h} .
\end{equation}
Therefore, combining \eqref{eq: R_1_bi}, \eqref{eq: f_bi} and \eqref{eq: R_2_bi}  we obtain that
\begin{equation} \label{eq: perturb_R1}
R_{1} \lesssim \varepsilon^{2} h (|u|_{2,\Omega} + |u|_{3,\Omega}) |w_h|_{2,h} \lesssim  \varepsilon h (|u|_{2,\Omega} + |u|_{3,\Omega} )\| w_{h}\|_{\varepsilon,h}
\end{equation}
where we utilize $\varepsilon|w_{h}|_{2,h} \leqslant \| w_{h} \|_{\varepsilon,h} $.

By $V_{h0}^{\rm{R}} \subset V_{hs}^{\rm{M}}$, $H^{2}_{0}(\Omega) \subset H^{1}_{0}(\Omega)$, and $|w_{h}|_{1,h}\lesssim \| w_{h}\|_{\varepsilon,h}$, we have,  by Lemma \ref{lem:consisRM}, that
\begin{equation} \label{eq: perturb_R2}
R_{2} \lesssim h|u|_{2,\Omega} |w_{h}|_{1,h} \lesssim h|u|_{2,\Omega}\| w_{h}\|_{\varepsilon,h}.
\end{equation}
Specially, if the mesh is uniform, then 
$
R_{2} \lesssim h^{2}|u|_{3,\Omega}\| w_{h}\|_{\varepsilon,h}.
$

For $R_{3}$, it holds that 
$$
R_{3} \lesssim \varepsilon^{2} h  \|\Delta(\beta(\boldsymbol{x})\Delta u)\|_{0,\Omega} |w_{h}|_{1,h}
\lesssim \varepsilon^2 h  (|u|_{2,\Omega} + |u|_{3,\Omega}+ \|\Delta^{2}u\|_{0,\Omega})\| w_{h}\|_{\varepsilon,h}.
$$
A combination of the estimates of terms $R_{1}$, $R_{2}$, and $R_{3}$ leads to
\begin{equation}\label{eq:consis nonuniform}
\sup_{w_{h}\in V_{h0}^{\rm{R}},w_{h} \ne 0}  \frac{E_{\varepsilon,h}(u,w_{h})}{\| w_{h}\|_{\varepsilon,h} } 
\lesssim h |u|_{2,\Omega}  + \varepsilon h (|u|_{2,\Omega} + |u|_{3,\Omega}) + \varepsilon^{2} h   (|u|_{2,\Omega} + |u|_{3,\Omega}+ \|\Delta^{2}u\|_{0,\Omega})
\end{equation}
and the following estimate on uniform subdivisions
\begin{equation}\label{eq:consis uniform}
\sup_{w_{h}\in V_{h0}^{\rm{R}},w_{h} \ne 0}  \frac{E_{\varepsilon,h}(u,w_{h})}{\| w_{h}\|_{\varepsilon,h} } 
\lesssim h^{2} |u|_{3,\Omega}  + \varepsilon h (|u|_{2,\Omega} + |u|_{3,\Omega})  + \varepsilon^{2} h   (|u|_{2,\Omega} + |u|_{3,\Omega}+ \|\Delta^{2}u\|_{0,\Omega}).
\end{equation}
 By \eqref{eq:StrangRRM}, \eqref{lem:approx in energy norm}, \eqref{eq:consis nonuniform}, and \eqref{eq:consis uniform}, the proof can be completed imediately.\qed
\subsubsection{\bf Proof of Lemma \ref{lem:regularity}}

Similar to the proof of \!\cite[Lemma~5.1]{Nilssen;Tai;Winther2001}, the lemma holds. By the regularity theory of second order elliptic equation, 
\begin{equation} \label{eq:regularity_second}
\left\|u^0\right\|_2 \leq c\|f\|_0
\end{equation}
and, since $\Delta(\beta(\boldsymbol{x})\Delta u)=\varepsilon^{-2} \Delta\left(u-u^0\right)$, it is a consequence of the regularity theory of fourth order elliptic equation that 
\begin{equation} \label{eq:regularity_fourth}
\|u\|_3 \leq c \varepsilon^{-2}\left\|\Delta\left(u-u^0\right)\right\|_{-1} \leq c \varepsilon^{-2}\left|\left(u-u^0\right)\right|_1.
\end{equation}

Furthermore, from the weak formulations of the problems \eqref{eq:model problem} and \eqref{eq:2nd problem}, and the fact that $u \in H_0^2 \cap H^3$, we derive that
$$
\varepsilon^2(\beta\Delta u, \Delta v)+\left(D\left(u-u^0\right), D v\right)=\varepsilon^2 \int_{\partial \Omega}(\beta \Delta u) \frac{\partial v}{\partial n} d s
$$
for all $v \in H_0^1 \cap H^2$. In particular, by choosing $v=u-u^0$ and the boundedness of $\beta$, we obtain
\begin{equation} \label{proof:lemma11_1}
\varepsilon^2\|\Delta u\|_0^2+\left|u-u^0\right|_1^2 \leq \varepsilon^2 \int_{\partial \Omega}|\Delta u \frac{\partial u^0}{\partial n}| d s + \varepsilon^2\int_{\Omega} |\Delta u \cdot f| d x .
\end{equation}
However,
\begin{equation} \label{proof:lemma11_2}
\varepsilon^2\int_{\Omega} |\Delta u \cdot f| d x \leq \frac{\varepsilon^2}{2}\left(\|\Delta u\|_0^2+\|f\|_0^2\right)
\end{equation}
Furthermore, standard trace inequalities and \eqref{eq:regularity_second} imply
$$
\int_{\partial \Omega}\left|\frac{\partial u^0}{\partial n}\right|^2 d s \leq c\|f\|_0^2
$$
and
$$
\int_{\partial \Omega}|\Delta u|^2 d s \leq c\|\Delta u\|_0\|\Delta u\|_1
$$
Hence, from the arithmetic geometric mean inequality we obtain that for any $\delta>0$ there is a constant $c_\delta$ such that
\begin{equation} \label{proof:lemma11_3}
\varepsilon^2\left|\int_{\partial \Omega}(\Delta u) \frac{\partial u^0}{\partial n} d s\right| \leq c_\delta \varepsilon\|f\|_0^2+\delta \varepsilon^3\|\Delta u\|_0\|u\|_3 .
\end{equation}
However, from \eqref{eq:regularity_fourth} we derive
\begin{equation} \label{proof:lemma11_4}
\begin{aligned}
\varepsilon^3\|\Delta u\|_0\|u\|_3 & \leq \frac{1}{2}\left(\left\|\varepsilon^2 \Delta u\right\|_0^2+\varepsilon^4\|u\|_3^2\right) \\
& \leq \frac{1}{2} \varepsilon^2\|\Delta u\|_0^2+c\left|u-u^0\right|_1^2 .
\end{aligned}
\end{equation}
The inequalities \eqref{proof:lemma11_1}-\eqref{proof:lemma11_4} lead to the bound
$$
\varepsilon^2\|\Delta u\|_0^2+\left|u-u^0\right|_1^2 \leq c \varepsilon\|f\|_0^2,
$$
and together with \eqref{eq:regularity_fourth} this implies the desired estimates.\qed

\subsubsection{\bf Proof of Theorem \ref{thm:uniform convergence}} 
\label{sec:pfuc} 
From Theorem~\ref{thm:approxH02}, we have
$\big|u-\Pi_{h0}u\big|_{2,h}^{2} \lesssim |u|_{2,\Omega}\big|u-\Pi_{h0}u\big|_{2,h}  \lesssim h |u|_{2,\Omega} |u|_{3,\Omega}.$ By Lemma~\ref{lem:regularity}, we further obtain
\begin{align}\label{eq:H2norm(u-Piuh)}
\varepsilon^{2}\big|u-\Pi_{h0}u\big|_{2,h}^{2} \lesssim h \varepsilon^{\frac{1}{2}}|u|_{2,\Omega} \varepsilon^{\frac{3}{2}}|u|_{3,\Omega}  \lesssim h \|f\|_{0,\Omega}^{2}.
\end{align}
From \cite[Theorem 3.2.1.2] {Grisvard1985}, $|u^{0}|_{2,\Omega} \lesssim \|f\|_{0,\Omega}$. This, together with Lemma ~\ref{lem:regularity}, leads to 
\begin{equation}\label{eq:H1norm(u-Piuh)}
\begin{split}
\big|u-\Pi_{h0}u\big|_{1,h}^{2}  & \lesssim \big|(u-u^{0})- \Pi_{h0}(u-u^{0})\big|_{1,h}^{2} + \big|u^{0}-\Pi_{h0}u^{0}\big|_{1,h}^{2} \\
& = \big|(u-u^{0})- \Pi_{h0}(u-u^{0})\big|_{1,h} \, \big|(u-u^{0})- \Pi_{h0}(u-u^{0})\big|_{1,h} + \big|u^{0}-\Pi_{h0}u^{0}\big|_{1,h}^{2} \\
 & \lesssim \big|u-u^{0}\big|_{1,\Omega}\ h\big|u-u^{0}\big|_{2,\Omega} + h^{2}\big|u^{0}\big|_{2,\Omega}^{2} \\
 & \lesssim h \big (\varepsilon^{-\frac{1}{2}} \big|u-u^{0}\big|_{1,\Omega}\big)\, \big (\varepsilon^{\frac{1}{2}} \big|u-u^{0}\big|_{2,\Omega}\big) + h^{2}\big|u^{0}\big|_{2,\Omega}^{2} \\
 & \lesssim h\|f\|_{0,\Omega}^{2} + h^{2}\|f\|_{0,\Omega}^{2}  \lesssim h\|f\|_{0,\Omega}^{2}.
 \end{split}
\end{equation}

\noindent From \eqref{eq:H2norm(u-Piuh)} and \eqref{eq:H1norm(u-Piuh)}, we obtain
\begin{align}\label{eq:unifrom approx error}
\inf_{v_{h}\in V_{h0}^{\rm{R}}} \| u - v_{h} \|_{\varepsilon,h} \leq \| u - \Pi_{h0}u \|_{\varepsilon,h} \lesssim h^{\frac{1}{2}}\|f\|_{0,\Omega}.
\end{align}

Next we are to estimate the consistency error. For term $R_{1}$:
$$
R_{1} = \varepsilon^{2}  \sum_{K \in \mathcal{G}_h} \int_K \nabla(\beta(\boldsymbol{x})\Delta u)\cdot (\nabla w_h - \nabla \Pi_{h0}^{\rm{b}} w_h) - \varepsilon^{2} \sum_{K \in \mathcal{G}_h} \int_{\partial K} \beta(\boldsymbol{x})\Delta u \frac{\partial w_h}{\partial \boldsymbol{n}} 
$$
It can be divided into two parts, for the first part, we have
\begin{equation}\label{eq:R1_1}
\begin{split}
\varepsilon^{2}  \sum_{K \in \mathcal{G}_h} \int_K \nabla(\beta(\boldsymbol{x})\Delta u)\cdot (\nabla w_h - \nabla \Pi_{h0}^{\rm{b}} w_h) 
 & \lesssim 
\varepsilon^{2} \sum_{K\in \mathcal{G}_{h}}  ( |u|_{2,K} + |u|_{3,K}) \, |\Pi_{h0}^{\rm{b}} w_{h} - w_{h}|_{1,K}^{\frac{1}{2}} \, |\Pi_{h0}^{\rm{b}} w_{h} - w_{h}|_{1,K}^{\frac{1}{2}}
 \\
 & \lesssim \varepsilon^{2} |( |u|_{2,K} + |u|_{3,K})\, |w_{h}|_{1,h}^{\frac{1}{2}}\, h^{\frac{1}{2}} |w_{h}|_{2,h}^{\frac{1}{2}}\\
& \lesssim \varepsilon^{\frac{3}{2}} h^{\frac{1}{2}} ( |u|_{2,K} + |u|_{3,K}) \| w_{h} \|_{\varepsilon,h},
\end{split}
\end{equation}
where we utilize $|w_{h}|_{1,h} \lesssim \| w_{h} \|_{\varepsilon,h}$ and $\varepsilon|w_{h}|_{2,h}\lesssim \| w_{h} \|_{\varepsilon,h}$.

\noindent Now we are to estimate another part of  $R_{1}$. Notice that $\int_{e} \llbracket\partial_{x} w_{h} \rrbracket_e \,d s =\int_{e} \llbracket\partial_{y} w_{h} \rrbracket_e \,d s = 0$ for any $e\in \mathcal{E}_{h}$. Moreover, from the trace theorem, $\|v\|_{0,e} \lesssim \|v\|_{0,K}^{\frac{1}{2}} |v|_{1,K}^{\frac{1}{2}}$ for $e\subset \partial K$. Then we have
 \begin{equation}\label{eq:R1_2}
\varepsilon^{2} \sum_{K \in \mathcal{G}_h} \int_{\partial K} \beta(\boldsymbol{x})\Delta u \frac{\partial w_h}{\partial \boldsymbol{n}} \lesssim \varepsilon^{2} |u|_{2,\Omega}^{\frac{1}{2}}\,|u|_{3,\Omega}^{\frac{1}{2}}\,h^{\frac{1}{2}}  | w_{h} |_{2,h} \lesssim \varepsilon h^{\frac{1}{2}} |u|_{2,\Omega}^{\frac{1}{2}}\,|u|_{3,\Omega}^{\frac{1}{2}} \| w_{h} \|_{\varepsilon,h}.
\end{equation}
Combining these two parts and utilizing Lemma~\ref{lem:regularity}, we obtain that
\begin{equation*}
\begin{split}
R_{1}  &\lesssim \big(\varepsilon^{\frac{3}{2}} h^{\frac{1}{2}}(|u|_{2,\Omega} + |u|_{3,\Omega}) + \varepsilon h^{\frac{1}{2}} |u|_{2,\Omega}^{\frac{1}{2}}\, |u|_{3,\Omega}^{\frac{1}{2}} \big)\| w_{h} \|_{\varepsilon,h} \\
& \lesssim h^{\frac{1}{2}}\, \big(\varepsilon\varepsilon^{\frac{1}{2}} |u|_{2,\Omega} + \varepsilon^{\frac{3}{2}} |u|_{3,\Omega} +  \varepsilon^{\frac{1}{4}} |u|_{2,\Omega}^{\frac{1}{2}}\, \varepsilon^{\frac{3}{4}}|u|_{3,\Omega}^{\frac{1}{2}} \big)\| w_{h} \|_{\varepsilon,h} \\
& \lesssim h^{\frac{1}{2}} (\varepsilon\|f\|_{0,\Omega} + \|f\|_{0,\Omega})\,\| w_{h} \|_{\varepsilon,h}.
\end{split}
\end{equation*}

Next we consider $R_{2}+R_{3}$ in the consistency error.
By \eqref{eq:model problem} and \eqref{eq:2nd problem},  $\varepsilon^{2} \Delta(\beta(\boldsymbol{x})\Delta u) = \Delta(u-u^{0})$.
When $h< \varepsilon$, we have, by Lemmas \ref{lem:consisRM} and \ref{lem:regularity}, that
\begin{equation*}
\begin{split}
R_{2}+R_{3} & =  \sum_{K\in \mathcal{G}_{h}} \int_{K} (\nabla u \cdot \nabla w_{h} + \Delta u \ w_{h}) + \varepsilon^{2}\int_{\Omega} \Delta(\beta(\boldsymbol{x})\Delta u) \ (\Pi_{h0}^{\rm{b}} w_{h} - w_{h})  \\
& =  \sum_{K\in \mathcal{G}_{h}} \int_{K} (\nabla u \cdot \nabla w_{h} + \Delta u \ w_{h}) + \int_{\Omega} \Delta(u-u^{0}) (\Pi_{h0}^{\rm{b}} w_{h} - w_{h}) \\
& \lesssim h|u|_{2,\Omega}|w_{h}|_{1,h} + h|u-u^{0}|_{2,\Omega}|w_{h}|_{1,h} \\
& \lesssim h^{\frac{1}{2}} \varepsilon^{\frac{1}{2}}|u|_{2,\Omega}|w_{h} |_{1,h} + h^{\frac{1}{2}} \varepsilon^{\frac{1}{2}}(|u|_{2,\Omega}+|u^{0}|_{2,\Omega}) |w_{h} |_{1,h}   \lesssim h^{\frac{1}{2}}\|f\|_{0,\Omega}\| w_{h} \|_{\varepsilon,h}.
\end{split}
\end{equation*}
When $\varepsilon \leq h$, noticing that $\Pi_{h0}^{\rm{b}} w_{h} \in H^{1}_{0}(\Omega)$, we obtain
\begin{equation*}
\begin{split}
R_{2}+R_{3} & =  \sum_{K\in \mathcal{G}_{h}} \int_{K} (\nabla u \cdot \nabla w_{h} + \Delta u \ w_{h}) + \int_{\Omega} \Delta(u-u^{0})(\Pi_{h0}^{\rm{b}} w_{h} - w_{h})  \\
& =  \sum_{K\in \mathcal{G}_{h}} \int_{K} (\nabla u \cdot \nabla (w_{h}-\Pi_{h0}^{\rm{b}} w_{h}) + \Delta u (w_{h}-\Pi_{h0}^{\rm{b}} w_{h})) - \int_{\Omega} \Delta(u-u^{0})(w_{h} -\Pi_{h0}^{\rm{b}} w_{h} )  \\
& =  \sum_{K\in \mathcal{G}_{h}} \int_{K} \nabla u \cdot \nabla (w_{h}-\Pi_{h0}^{\rm{b}} w_{h})  + \int_{\Omega} \Delta u^{0} (w_{h} -\Pi_{h0}^{\rm{b}} w_{h} ) \\
& = \sum_{K\in \mathcal{G}_{h}} \int_{K} \nabla (u-u^{0}) \cdot \nabla (w_{h}-\Pi_{h0}^{\rm{b}} w_{h}) + \sum_{K\in \mathcal{G}_{h}} \int_{\partial K} \frac{\partial u^{0}}{\partial \mathbf{n}} (w_{h}-\Pi_{h0}^{\rm{b}} w_{h}) \\
& \lesssim \varepsilon^{-\frac{1}{2}} |u-u^{0}|_{1,\Omega}  \varepsilon^{\frac{1}{2}} |w_{h}|_{1,h} + h|u^{0}|_{2,\Omega}|w_{h}|_{1,h} \\
& \lesssim \|f\|_{0,\Omega}h^{\frac{1}{2}} |w_{h}|_{1,h}  + h\|f\|_{0,\Omega} |w_{h}|_{1,h} \lesssim h^{\frac{1}{2}}\|f\|_{0,\Omega} \| w_{h} \|_{\varepsilon,h},
\end{split}
\end{equation*}
where Lemma~\ref{lem:consisRM} is utilized to estimate the term $\sum\limits_{K\in \mathcal{G}_{h}} \int_{\partial K} \frac{\partial u^{0}}{\partial \mathbf{n}} (w_{h}-\Pi_{h0}^{\rm{b}} w_{h})$.
Hence we obtain
\begin{align}\label{eq:unifrom consis error}
E_{\varepsilon,h}(u,w_{h}) = R_{1}+R_{2}+R_{3}  \lesssim h^{\frac{1}{2}} (\varepsilon\|f\|_{0,\Omega} + \|f\|_{0,\Omega})\,\| w_{h} \|_{\varepsilon,h}.
\end{align}
Combining \eqref{eq:StrangRRM}, \eqref{eq:unifrom approx error}, and \eqref{eq:unifrom consis error}, we obtain the uniform convergence.\qed


\end{document}